\documentclass[12pt,a4paper,reqno]{amsart}
\usepackage{amsmath,amssymb,fullpage}
\usepackage{mathrsfs}
\usepackage{esint}  
\usepackage{color}
\usepackage{graphicx}
\usepackage{psfrag}
\usepackage{enumerate}

\author{Michael Feischl}
\author{Gregor Gantner}
\author{Dirk Praetorius}

\address{Vienna University of Technology,
 Institute for Analysis and Scientific Computing,
 Wiedner Hauptstra\ss{}e 8-10,
 A-1040 Wien, Austria}
\email{\{\,Michael.Feischl\,,\,Dirk.Praetorius\,\}@tuwien.ac.at}
\email{Gregor.Gantner@tuwien.ac.at\quad\rm(corresponding author)}


\title{Reliable and efficient a~posteriori error estimation\\ for adaptive IGA boundary element methods\\ for weakly-singular integral equations}

\def\H{\widetilde{H}}

\def\N{\mathbb{N}}
\def\R{\mathbb{R}}
\def\Z{\mathbb{Z}}

\def\KK{\mathcal{K}}
\def\NN{\mathcal{N}}

\def\TT{\mathcal{T}}

\def\supp{{\rm supp}}

\newcommand{\set}[3][\big]{#1\{#2\,:\,#3#1\}}
\newcommand{\norm}[3][]{#1\|#2#1\|_{#3}}
\def\dual#1#2{\langle#1\,;\,#2\rangle}

\def\enorm#1{|\!|\!|#1|\!|\!|}

\numberwithin{equation}{section}
\numberwithin{figure}{section}
\newtheorem{theorem}{Theorem}[section]
\newtheorem{proposition}[theorem]{Proposition}
\newtheorem{lemma}[theorem]{Lemma}

\newtheorem{algorithm}[theorem]{Algorithm}
\newtheorem{remark}[theorem]{Remark}

\newcounter{const}

\date{\today}

\def\MM{\mathcal M}
\def\XX{\mathcal X}
\def\Crel{C_{\rm rel}}
\def\Ceff{C_{\rm eff}}

\begin{document}
\maketitle
\begin{abstract}
We consider the Galerkin boundary element method (BEM) for weakly-singular integral equations of the first-kind in 2D.
We analyze some residual-type a~posteriori error estimator which provides a lower as well 
as an upper bound for the unknown Galerkin BEM error. The required assumptions are weak and allow for piecewise smooth 
parametrizations of the boundary, local mesh-refinement, and related standard piecewise polynomials as well as NURBS. 
In particular, our analysis gives a first contribution to adaptive BEM in the frame of isogeometric analysis (IGABEM), for which we formulate an adaptive algorithm which steers the local mesh-refinement and the multiplicity of the knots. 
Numerical experiments underline the theoretical findings and show that the proposed adaptive strategy leads to optimal convergence.
\end{abstract}

\section{Introduction} 


\noindent
{\bf Isogeometric analysis.}\quad
The central idea of isogeometric analysis is to use the same ansatz functions for
the discretization of the partial differential equation at hand, as are used for 
the representation of the problem geometry. Usually, the problem geometry $\Omega$ is 
represented in computer aided design (CAD) by means of NURBS or T-splines. This 
concept, originally invented in~\cite{hughes2005}
for finite element methods (IGAFEM) has proved very fruitful in 
applications~\cite{hughes2005,simpson2012}; see also the monograph \cite{bible}. 
Since CAD directly provides a parametrization of the boundary $\partial \Omega$, this makes 
the boundary element method (BEM) the most attractive numerical scheme, if applicable
(i.e., provided that the fundamental solution of the differential operator is explicitly known).
Isogeometric BEM (IGABEM) has first been considered in~\cite{simpson2012}.
Unlike standard BEM with piecewise polynomials which is well-studied in the literature,
cf.~the monographs~\cite{ss,steinbach} and the references therein,
the numerical analysis of IGABEM is essentially open. We only refer 
to~\cite{simpson2012,helmholtziga,laplaceiga} for numerical experiments and to~\cite{stokesiga} for some quadrature analysis.
In particular, a~posteriori error estimation has been well-studied for standard BEM,
e.g., \cite{cs95,cs96,cc97,cmps,cms,faermann2d,faermann3d} as well as the recent overview article~\cite{abem}, but has not
been treated for IGABEM so far. The purpose of the present work is to shed some first
light on a~posteriori error analysis for IGABEM which provides some mathematical
foundation of a corresponding adaptive algorithm.

\bigskip

\noindent
{\bf Main result.}\quad
Let $\Omega\subset\R^2$ be a Lipschitz domain and $\Gamma\subseteq \partial\Omega$ be a compact, piecewise smooth part of the boundary with finitely many connected components (see Section~\ref{section:dpr} and Section~\ref{subsec:boundary parametrization}).
Given a right-hand side $f$, 
we consider boundary integral equations in the abstract form
\begin{align}\label{eq:strong}
 V\phi(x) = f(x)
 \quad\text{for all }x\in\Gamma_{},
\end{align}
where $V:\H^{-1/2}(\Gamma_{})\to H^{1/2}(\Gamma_{})$ is an elliptic isomorphism.
Here $H^{1/2}(\Gamma_{})$ is a fractional-order Sobolev space, and
$\H^{-1/2}(\Gamma_{})$ is its dual (see Section~\ref{section:preliminaries} below).
Given $f\in H^{1/2}(\Gamma_{})$, the Lax-Milgram lemma provides existence
and uniqueness of the solution $\phi\in\H^{-1/2}(\Gamma_{})$ of the variational formulation of~\eqref{eq:strong}
\begin{align}\label{eq:weak}
 \int_{\Gamma_{}} V\phi(x)\psi(x)\,dx
 = \int_{\Gamma_{}} f(x)\psi(x)\,dx
 \quad\text{for all }\psi\in\H^{-1/2}(\Gamma_{}).
\end{align}
In the Galerkin boundary element method (BEM), the test
space $\H^{-1/2}(\Gamma_{})$ is replaced by some discrete subspace  $\XX_h\subseteq {L^{2}(\Gamma_{})}\subseteq\H^{-1/2}(\Gamma_{})$.
Again, the Lax-Milgram lemma guarantees existence and uniqueness of the solution
$\phi_h\in\XX_h$ of the discrete variational formulation
\begin{align}\label{eq:discrete}
 \int_{\Gamma_{}} V\phi_h(x)\psi_h(x)\,dx
 = \int_{\Gamma_{}} f(x)\psi_h(x)\,dx
 \quad\text{for all }\psi_h\in\XX_h,
\end{align}
and $\phi_h$ can in fact be computed by solving a linear system of equations.

We assume that $\XX_h$ is linked with a partition $\TT_h$ of $\Gamma$ into a set of connected segments. For each
vertex ${x}\in\NN_h$ of $\TT_h$, let $\omega_h({z}) := \bigcup\set{T\in\TT_h}{{z}\in T}$ 
denote the node patch. If $\XX_h$ is sufficiently rich (e.g., $\XX_h$ contains 
certain splines or NURBS; see Section~\ref{section:abem}), we prove that
\begin{align}\label{eq:faermann}
 \Crel^{-1}\,\norm{\phi-\phi_h}{\H^{-1/2}(\Gamma_{})}
 \le \eta_h:=\Big(\sum_{{z}\in\NN_h}|r_h|_{H^{1/2}(\omega_h({z}))}^2\Big)^{1/2}
 \le \Ceff\,\norm{\phi-\phi_h}{\H^{-1/2}(\Gamma_{})}
\end{align}
with some $\XX_h$-independent constants $\Ceff,\Crel>0$, i.e., the unknown BEM 
error  is controlled by some computable a~posteriori error estimator
$\eta_h$. Here, $r_h:=f-V\phi_h \in H^{1/2}(\Gamma_{})$ denotes the residual and
\begin{align}
 |r_h|_{H^{1/2}(\omega_h({z}))}
 := \int_{\omega_h({z})}\int_{\omega_h({z})}\frac{|r_h(x)-r_h(y)|^2}{|x-y|^2}\,dy\,dx
\end{align}
is the Sobolev-Slobodeckij seminorm. 

Estimate~\eqref{eq:faermann} has first been
proved by Faermann~\cite{faermann2d} for closed $\Gamma_{}=\partial\Omega$ and standard spline spaces $\XX_h$ based on the
arclength parametrization $\gamma:[0,L] \to \Gamma$. In isogeometric analysis, 
$\gamma$ is \emph{not} the arclength parametrization. In our contribution, we 
generalize and refine the original analysis of Faermann~\cite{faermann2d}: Our analysis 
allows, first, closed as well as open parts of the boundary, second, general piecewise smooth parametrizations $\gamma$ and, third, 
covers standard piecewise polynomials as well as NURBS spaces $\XX_h$.

\bigskip

\noindent
{\bf Outline.}\quad
Section~\ref{section:preliminaries} recalls the functional analytic framework,
provides the assumptions on $\Gamma$ and its parametrization $\gamma$, and fixes
the necessary notation. The proof of~\eqref{eq:faermann} is given in 
Section~\ref{section:aposteriori} for sufficiently rich spaces $\XX_h$ {(Theorem \ref{thm:faermann})}. 
In Section~\ref{section:abem}, we recall
the NURBS spaces for IGABEM and prove that these spaces $\XX_h$ satisfy the assumptions {(Assumptions {\rm(A1)--(A2)} in Section \ref{sec:main theorem})} of 
the a~posteriori error estimate~\eqref{eq:faermann}. Based on knot insertion, 
we formulate an adaptive algorithm which is capable to control and adapt the 
multiplicity of the nodes as well as the local mesh-size {(Algorithm \ref{the algorithm})}. The final 
Section~\ref{section:numerics} gives some brief comments on the stable implementation
of adaptive IGABEM {for Symm's integral equation} and provides the numerical evidence for the superiority
of the proposed adaptive IGABEM over IGABEM with uniform mesh-refinement.


\section{Preliminaries}
\label{section:preliminaries}
\noindent{The purpose of this section is to collect the main assumptions on the boundary
and its discretization as well as to fix the notation. 
For more details on Sobolev spaces and the functional analytic setting of weakly-singular integral equations, we refer to the literature, e.g., the monographs \cite{hsiao,mclean,ss} and the references therein.}


Throughout, $|\cdot|$ denotes the absolute value of scalars, the Euclidean norm of vectors in $\R^2$, the measure of a set in $\R$, e.g. the length of an interval, or the arclength of a curve in $\R^{2}$.
The respective meaning will be clear from the context.

\def\Cgamma{C_\Gamma}
\subsection{Sobolev spaces}
For any measurable subset $\omega\subseteq\Gamma$, let $L^2(\omega)$ denote
the Lebesgue space of all square integrable functions which is associated with the norm
$\norm{u}{L^2(\omega)}^2:=\int_\omega |u(x)|^2\,dx$.
We define the Hilbert space
\begin{align}
 H^{1/2}(\omega) := \set{u\in L^2(\omega)}{\norm{u}{H^{1/2}(\omega)}<\infty},
\end{align}
associated with the Sobolev-Slobodeckij norm
\begin{align}\label{eq:SS-norm}
 \norm{u}{H^{1/2}(\omega)}^2
 := \norm{u}{L^2(\omega)}^2
 + |u|_{H^{1/2}(\omega)}^2
 \quad\text{with}\quad
 |u|_{H^{1/2}(\omega)}^2 := \int_\omega\int_\omega\frac{|u(x)-u(y)|^2}{|x-y|^2}\,dy\,dx.
\end{align}
For finite intervals $I\subseteq \R$ we use analogous definitions.
By $\H^{-1/2}(\omega)$, we denote the dual space of $H^{1/2}(\omega)$, where duality
is understood with respect to the $L^2(\omega)$-scalar product, i.e.,
\begin{align}
 \dual{u}{\phi} = \int_\omega u(x)\phi(x)\,dx
 \quad\text{for all }u\in H^{1/2}(\omega)
 \text{ and }\phi\in L^2(\omega).
\end{align}
{We note that $H^{1/2}(\Gamma)\subseteq L^2(\Gamma)\subseteq \H^{-1/2}(\Gamma)$ form a Gelfand triple and all inclusions are dense and compact.}
Amongst other equivalent definitions of $H^{1/2}(\omega)$ are the characterization
as trace space of functions {in} $H^1(\Omega)$ as well as equivalent interpolation techniques.
All these definitions provide the same space of functions but different norms, where
norm equivalence constants depend only  on $\omega$; see, e.g., the monograph~\cite{mclean}
and references therein.
Throughout, we shall use the Sobolev-Slobodeckij
norm \eqref{eq:SS-norm}, since it is  numerically computable.

\subsection{Connectedness of $\Gamma$}\label{section:dpr}
Let the part of the boundary $\Gamma = \bigcup_i\Gamma_i$ be decomposed into its finitely many connected
components $\Gamma_i$. The $\Gamma_i$ are compact and piecewise
smooth as well. Note that this yields existence of some constant $c>0$ such
that $|x-y|\ge c > 0$ for all $x\in\Gamma_i$, $y\in\Gamma_j$, and $i\neq j$.
Together with $|u(x)-u(y)|^2 \le 2\,|u(x)|^2 + 2\,|u(y)|^2$, this provides the 
estimate
\begin{align*}
 \sum_{{i,j}\atop{i\neq j}} \int_{\Gamma_i}\int_{\Gamma_j}
 \frac{|u(x)-u(y)|^2}{|x-y|^2}\,dy\,dx
 \lesssim \sum_{i}\norm{u}{L^2(\Gamma_i)}^2
 + \sum_{j}\norm{u}{L^2(\Gamma_j)}^2
 \simeq \norm{u}{L^2(\Gamma)}^2
\end{align*}
and results in norm equivalence
\begin{align*}
 \|u\|^2_{H^{1/2}(\Gamma)}
 = \sum_{i}\|u\|^2_{H^{1/2}(\Gamma_i)} 
 + \sum_{{i,j}\atop{i\neq j}} \int_{\Gamma_i}\int_{\Gamma_j}\frac{|u(x)-u(y)|^2}{|x-y|^2}\,dy\,dx
 \simeq \sum_{i} \|u\|^2_{H^{1/2}(\Gamma_i)}.
\end{align*}
The usual piecewise polynomial and NURBS basis functions have connected support
and are hence supported by some \emph{single} $\Gamma_i$ each. Without loss of generality
and for the ease of presentation, 
we may therefore from now on assume that $\Gamma$ is connected. All results of this
work remain valid for non-connected $\Gamma$.
\color{black}

\subsection{Boundary parametrization}
\label{subsec:boundary parametrization}
We assume that either $\Gamma=\partial\Omega$ is parametrized by a closed continuous and
piecewise {two times} continuously differentiable path $\gamma:[a,b]\to\Gamma$ such
that the restriction $\gamma|_{[a,b)}$ is even bijective, or that $\Gamma\subsetneqq\partial\Omega$ is parametrized by a bijective continuous and piecewise two times continuously differentiable path $\gamma:[a,b]\to\Gamma$.  In the first case, we speak of \textit{closed} $\Gamma=\partial\Omega$, whereas the second case is referred to as \textit{open} $\Gamma\subsetneqq\partial\Omega$.
For closed $\Gamma$, we denote the $(b-a)$-periodic extension to $\R$ also by $\gamma$.
For the left and right derivative of $\gamma$, we assume that {$\gamma^{\prime_\ell}(t)\neq 0$ for $t\in(a,b]$ and $\gamma^{\prime_r}(t)\neq 0$  for $t\in [a,b)$.}
Moreover we assume that $\gamma^{\prime_\ell}(t)
+c\gamma^{\prime_r}(t)\neq0$ for all $c>0$ {and $t\in[a,b]$ resp. $t\in(a,b)$.} 
Finally, let $\gamma_L:[0,L]\to\Gamma$ denote the arclength parametrization, i.e.,
$|\gamma_L^{\prime_\ell}(t)| = 1 = |\gamma_L^{\prime_r}(t)|$, and its periodic extension. Then, elementary
differential geometry yields bi-Lipschitz continuity
\begin{align}\label{eq:bi-Lipschitz}
 \Cgamma^{-1} \le \frac{|\gamma_L(s)-\gamma_L(t)|}{|s-t|}\le\Cgamma
 \quad\text{for }s,t\in\R, {\text{ with }\begin{cases}
 |s-t|\le \frac{3}{4}\,L, \text{ for closed }\Gamma,\\
  s\neq t\in [0,L], \text{ for open }\Gamma.
\end{cases}}
\end{align}
A proof is given in \cite[Lemma 2.1]{diplarbeit} for closed $\Gamma$. 
For open $\Gamma$, the proof is even simpler. 
If $\Gamma$ is closed and $|I|\le \frac{3}{4} L$ resp. if $\Gamma$ is open and $I\subseteq [a,b]$,  we see from \eqref{eq:bi-Lipschitz} that
\begin{align}\label{eq:equivalent Hsnorm}
\Cgamma^{-1}|u\circ\gamma_{L}|_{H^{1/2}(I)}\leq |u|_{H^{1/2}(\gamma_L(I))}\leq \Cgamma|u\circ\gamma_{L}|_{H^{1/2}(I)}.
\end{align}

\subsection{Boundary discretization}
The part of the boundary $\Gamma$ is split into a set $\mathcal{T}_h=\{T_1,\dots,T_n\}$ of compact and connected segments $T_j$.
The endpoints of the elements of $\TT_h$ form the set of nodes $\mathcal{N}_h:=\set{z_j}{j=1,\dots,n}$ for closed $\Gamma$ and $\mathcal{N}_h=\set{z_j}{j=0,\dots,n}$ for open $\Gamma$.
The arclength of each element $T\in \mathcal{T}_h$ is denoted by $h_{T}$, where $h:=\max_{T\in\TT_h} h_T$.
Moreover, we define the \textit{shape regularity constant} 
\begin{align*}
\kappa(\mathcal{T}_h)&:=\max\Big(\set{h_{T}/h_{T'}}{T,T'\in\mathcal{T}_h, T\cap T'\neq \emptyset}\Big)
\end{align*}
For closed $\Gamma$,  we extend the nodes, elements and their length periodically. We suppose
\begin{align}\label{eq:h small}
h \le |\Gamma|/4,
\end{align}
if $\Gamma$ is closed.
\subsection{Parameter domain discretization}
Given the parametrization $\gamma:[a,b]\to \Gamma$, the discretization $\TT_h$ induces a discretization $\check{\TT}_h=\{\check{T}_1,\dots,\check{T}_n\}$ on the parameter domain $[a,b]$.
Let $a=\check{z}_0<\check{z}_1<\dots<\check{z}_n$ be the endpoints of the elements of $\check{\TT}_h$. 
We assume $\check{T}_j=[\check{z}_{j-1},\check{z}_j]$, $\gamma(\check{T_j})=T_j$ and $\gamma(\check{z}_j)=x_j$. 
We define $\check{\mathcal{N}}_h:=\set{\check{z}_j}{j=1,\dots,n}$ for closed $\Gamma=\partial\Omega$, and $\check{\mathcal{N}}_h:=\set{\check{z}_j}{j=0,\dots,n}$ for open $\Gamma\subsetneqq \partial\Omega$.
The length of each element $\check{T}\in \check{\TT}_h$ is denoted by $h_{\check{T}}$.
Moreover, we define the \textit{shape regularity constant} on $[a,b]$ as 
\begin{align*}
\kappa(\check{\mathcal{T}_h})&:=\max\Big(\set{h_{\check{T}}/h_{\check{T}'}}{\check{T},\check{T}'\in\check{\mathcal{T}_h}, \gamma(\check{T})\cap \gamma(\check{T}')\neq \emptyset}\Big).
\end{align*}


\section{A~posteriori error estimate}
\label{section:aposteriori}


\subsection{Main theorem}\label{sec:main theorem}
For $T\in\TT_h$, we inductively define the patch $\omega_h^m(T)\subseteq \Gamma$ of order $m\in\N_0$ by
\begin{align}\label{eq:patch}
 \omega_h^0(T) := T,\quad
 \omega_h^{m+1}(T) := \bigcup\set{T'\in\TT_h}{T'\cap\omega_h^m(T)\neq\emptyset}.
\end{align}
The main result of {Theorem \ref{thm:faermann}} requires the following two assumptions on $\TT_h$ and $\XX_h$
for some fixed integer $m\in\N_0$:
\begin{itemize}
\item[(A1)] For each $T\in\TT_h$, there exists some fixed function $\psi_T\in\XX_h$ with 
connected support $\supp(\psi_T)$ such that
\begin{align}\label{eq:phiT}
 T \subseteq \supp(\psi_T) \subseteq \omega_h^m(T).
\end{align}
\item[(A2)] There exists some constant $q\in (0,1]$ such that
\begin{align}\label{eq:contraction}
 \norm{1-\psi_T}{L^2(\supp(\psi_T))}^2 \le (1-q)\, |\supp(\psi_T)|
 \quad\text{for all }T\in\TT_h.
\end{align}
\end{itemize}
With these assumptions, we can formulate the following theorem which states 
validity of~\eqref{eq:faermann}. For standard BEM and piecewise polynomials based on the arclength parametrization $\gamma_{L}$ of some closed boundary $\Gamma=\partial\Omega$, the analogous result is first proved in \cite[Theorem 3.1]{faermann2d}
\begin{theorem}\label{thm:faermann}
The residual $r_h=f-V\phi_h$ satisfies { the efficiency estimate}
\begin{align}\label{eq1:thm:faermann}
 \eta_h:=\Big(\sum_{{z}\in\NN_h}|r_h|_{H^{1/2}(\omega_h({z}))}^2\Big)^{1/2}
 \le \Ceff\,\norm{\phi-\phi_h}{\H^{-1/2}(\Gamma_{})}.
\end{align}
If the mesh $\TT_h$ and the discrete space $\XX_h$ satisfy 
assumptions~{\rm(A1)--(A2)}, also the reliability estimate
\begin{align}\label{eq2:thm:faermann}
\norm{\phi-\phi_h}{\H^{-1/2}(\Gamma_{})}
\le \Crel\,\eta_h
\end{align}
holds.
The constant $\Ceff>0$ depends only on $V$, while $\Crel>0$ 
holds additionally on $\Gamma_{}$, $m$, $\kappa(\TT_h)$, and $q$.
\end{theorem}

\begin{remark}
The proof reveals that the efficiency estimate~\eqref{eq1:thm:faermann} is
valid for \emph{any} approximation $\phi_h$ of $\phi$, while the upper
reliability estimate~\eqref{eq2:thm:faermann} requires some Galerkin orthogonality.
\end{remark}
\subsection{Proof of efficiency estimate~(\ref{eq1:thm:faermann})}\quad
The elementary proof of the following proposition is already found in \cite[page 208]{faermann2d}. 
It is found as well in \cite[Theorem 2.12]{diplarbeit}.
\begin{proposition}\label{prop:efficient 2}
For each $u\in H^{1/2}(\Gamma_{})$, it holds
\begin{align}\label{eq:efficiency}
\sum_{{z}\in\NN_h}|u|_{H^{1/2}(\omega_h({z}))}^2
\le 2\,\norm{u}{H^{1/2}(\Gamma_{})}^2.
\end{align}
\end{proposition}

\begin{proof}[Proof of Theorem~\ref{thm:faermann}, eq.~\eqref{eq1:thm:faermann}]
Since $V$ is an isomorphism, the residual $r_{h}=f-V\phi_h = V(\phi-\phi_h)$
satisfies $\norm{r_h}{H^{1/2}(\Gamma_{})}\simeq \norm{\phi-\phi_h}{\H^{-1/2}(\Gamma_{})}$,
where the hidden constants depend only on $V$. Together 
with~\eqref{eq:efficiency}, this proves~\eqref{eq1:thm:faermann}.
\end{proof}

\subsection{Proof of reliability estimate~(\ref{eq2:thm:faermann})}
\quad  We start with the following lemma.
{For the elementary (but long) proof, we refer to \cite[Lemma 2.3]{faermann2d}.
A detailed proof is also found in \cite[Proposition 2.13]{diplarbeit}.}
\begin{lemma}\label{lem:Hnorm le Faer plus}
There exists a constant $C_1>0$ such that for all $u\in H^{1/2}(\Gamma_{})$
\begin{align*}
\norm{u}{H^{1/2}(\Gamma_{})}^2\leq \sum_{{z}\in \mathcal{N}_h} |u|_{H^{1/2}(\omega_h({z}))}^2+C_1\sum_{T\in\mathcal{T}_h} h_T^{-1}\norm{u}{L^2(T)}^2,
\end{align*}
The constant only depends on $\Gamma_{}$ and $\kappa(\TT_h)$.
\end{lemma}

Our next goal is to bound $\sum_{T\in\mathcal{T}_h} h_T^{-1}\norm{u}{L^2(T)}^2$.
To this end we need the following Poincar\'{e}-type inequality from \cite[Lemma 2.5]{faermann2d}.

\begin{lemma}\label{lem:Poincare}
Let $I\subset\R$ be a finite interval  with length $|I|>0$. 
Then, there holds
\begin{align*}
\norm{u}{L^2(I)}^2\leq \frac{|I|}{2}|u|_{H^{1/2}(I)}^2+\frac{1}{|I|}\left|\int_I u(t)\,dt \right|^2{\quad \text{for all }u\in L^2(I)}.
\end{align*}
\end{lemma}

\begin{lemma}\label{lem:Lnorm le h Hsnorm}
Suppose the assumptions~{\rm(A1)--(A2)}. 
Let $u\in H^{1/2}(\Gamma)$ satisfy
\begin{align}\label{eq:orthogonal}
 \int_{\Gamma_{}} u(x)\psi_T(x)\,dx = 0 \quad\text{for all }T\in\TT_h.
\end{align}
Then, there exists a constant $C_2>0$ which depends only on $\Gamma_{}$, $m$, 
$\kappa(\TT_h)$, and $q$ such that for all $T\in \mathcal{T}_h$
\begin{align}
\begin{aligned}
\norm{u}{L^2(T)}^2 & \leq C_2 h_T|u|_{H^{1/2}(T)}^2 \quad &\text{if } m=0,\\
\norm{u}{L^2(\mathrm{supp}(\psi_T))}^2&\le C_2|\mathrm{supp}(\psi_T)| \sum_{{z} \in \omega_h^{m-1}(T)\cap \mathcal{N}_h}|u|_{H^{1/2}(\omega_h({z}))}^2 \quad &\text{if } m>0.
\end{aligned}
\end{align}
\end{lemma}
\begin{proof}[Proof of Lemma \ref{lem:Lnorm le h Hsnorm} for closed $\Gamma_{}=\partial\Omega$]
The assertion is formulated on the boundary itself. 
Without loss of generality, we may therefore assume that $\gamma=\gamma_L$.
 Since $\mathrm{supp}(\psi_T)$ is connected, there is an interval $I$ of length $|I|\le L$ with $\gamma(I)=\mathrm{supp}(\psi_T)$.  
We use Lemma~\ref{lem:Poincare} and get
\begin{equation*}
\norm{u\circ\gamma}{L^2(I)}^2\leq\frac{|I|}{2}|u\circ\gamma|_{H^{1/2}(I)}^2+\frac{1}{|I|}\left|\int_{I}{u\circ\gamma(t)}\,d{t}\right |^2.
\end{equation*}
 With the orthogonality \eqref{eq:orthogonal} and  Assumption {\rm(A2)}, we  see
\begin{align*}
\left|\int_{I}{u\circ\gamma(t)}\,d{t}\right |^2 &= \left|\int_{\mathrm{supp}(\psi_T)}{u(y)(1-\psi_T(y))}\,d{y}\right|^2
=\left|\int_{I}{\big(u\circ\gamma(t)\big)\big(1-\psi_T\circ\gamma(t)\big)}\,d{t}\right|^2\\
&\leq \norm{1-(\psi_T\circ\gamma)}{L^2(I)}^2\norm{u\circ\gamma}{L^2(I)}^2 
\leq(1-q) |I| \norm{u\circ\gamma}{L^2(I)}^2.
\end{align*}
Using the {last two inequalities}, we therefore get
\begin{equation*}
\norm{u\circ\gamma}{L^2(I)}^2\leq\frac{|I|}{2} |{u\circ\gamma}|_{H^{1/2}(I)}^2+(1-q)\norm{u\circ\gamma}{L^2(I)}^2.
\end{equation*}
{Together with $|I|=|\gamma(I)|=|\mathrm{supp} (\psi_T)|$, this implies}
\begin{equation}\label{eq:Lnorm patch leq Hseminorm[1/2] I}
\norm{u}{L^{2}(\mathrm{supp}(\psi_T))}^2\leq \frac{|\mathrm{supp}(\psi_T)|}{2q}|{u\circ\gamma}|_{H^{1/2}(I)}^2.
\end{equation}
{For $m=0$, \eqref{eq:h small}, {\rm (A1)} and \eqref{eq:equivalent Hsnorm}, \eqref{eq:Lnorm patch leq Hseminorm[1/2] I} conclude the proof with $C_2=\Cgamma^2/2q$}.
To estimate $|u\circ\gamma|_{H^{1/2}(I)}^2$ for $m>0$, we use induction on $\ell$ to prove the following assertion for all $\ell \in \N$:
\begin{align}\label{eq:Hsnorm le Hsnorm local ass}
\forall j\in\Z \quad |{u\circ\gamma}|_{H^{1/2}([\check{{z}}_{j-1},\check{{z}}_{j+\ell}])}^2\leq (1+2\kappa(\mathcal{T}_h))^{\ell-1}\sum_{k=j}^{j+\ell-1}|{u\circ\gamma}|_{H^{1/2}(\check{T}_{k}\cup\check{T}_{k+1})}^2.
\end{align}
For $\ell=1$, {\eqref{eq:Hsnorm le Hsnorm local ass} even holds with equality}. 
The induction hypothesis for $\ell-1\geq 1$ is 
\begin{equation}\label{eq:Hsnorm le Hsnorm local hyp}
\forall j\in \Z   \quad |{u\circ\gamma}|_{H^{1/2}([\check{{z}}_{j-1},\check{{z}}_{j+\ell-1}])}^2\leq (1+2\kappa(\mathcal{T}_h))^{\ell-2}\sum_{k=j}^{j+\ell-2}|{u\circ\gamma}|_{H^{1/2} (\check{T}_{k}\cup\check{T}_{k+1})}^2.
\end{equation}
For $r,s \in \R$, let
\begin{equation*}
\check{U}(r,s):=\frac{|u(\gamma(r))-u(\gamma(s))|^2}{|r-s|^2}.
\end{equation*}
For $j\in \Z$, {the definition of the Sobolev-Slobodeckij seminorm \eqref{eq:SS-norm} shows}{
\begin{align}\label{eq:ind step Hseminorm[1/2] leq}
\begin{split}
&|{u\circ\gamma}|_{H^{1/2}([\check{{z}}_{j-1},\check{{z}}_{j+\ell}])}^2  =  \int_{[\check{{z}}_{j-1},\check{{z}}_{j+\ell-1}]}\int_{[\check{{z}}_{j-1},\check{{z}}_{j+\ell-1}]}{\check{U}(r,s)}\,d{r}\,d{s}  \\
&\quad\quad+ \int_{[\check{{z}}_{j+\ell-1},\check{{z}}_{j+\ell}]}\int_{[\check{{z}}_{j+\ell-1},\check{{z}}_{j+\ell}]}{\check{U}(r,s)}\,d{r}\,d{s}  +  2\int_{[\check{{z}}_{j+\ell-1},\check{{z}}_{j+\ell}]}\int_{[\check{{z}}_{j-1},\check{{z}}_{j+\ell-1}]}{\check{U}(r,s)}\,d{r}\,d{s}\\
&\quad=|{u\circ\gamma}|_{H^{1/2}([\check{{z}}_{j-1},\check{{z}}_{j+\ell-1}])}^2+|{u\circ\gamma}|_{H^{1/2}([\check{{z}}_{j+\ell-1},\check{{z}}_{j+\ell}])}^2  \\
&\quad\quad+ 2\int_{[\check{{z}}_{j+\ell-1},\check{{z}}_{j+\ell}]}\int_{[\check{{z}}_{j+\ell-2},\check{{z}}_{j+\ell-1}]}{\check{U}(r,s)}\,d{r}\,d{s}  +
2\int_{[\check{{z}}_{j+\ell-1},\check{{z}}_{j+\ell}]}\int_{[\check{{z}}_{j-1},\check{{z}}_{j+\ell-1}]}{\check{U}(r,s)}\,d{r}\,d{s}\\
&\quad\leq|{u\circ\gamma}|_{H^{1/2}([\check{{z}}_{j-1},\check{{z}}_{j+\ell-1}])}^2+|{u\circ\gamma}|_{H^{1/2}([\check{{z}}_{j+\ell-2},\check{{z}}_{j+\ell}])}^2\\
&\quad\quad+2\int_{[\check{{z}}_{j+\ell-1},\check{{z}}_{j+\ell}]}\int_{[\check{{z}}_{j-1},\check{{z}}_{j+\ell-1}]}{\check{U}(r,s)}\,d{r}\,d{s}.
\end{split}
\end{align}}
For $r <t <s \in \R$ , we have 
\begin{equation*}
\check{U}(r,s)\leq 2\frac{|u(\gamma(r))-u(\gamma(t))|^2}{|r-s|^2}+2\frac{|u(\gamma(t))-u(\gamma(s))|^2}{|r-s|^2}\leq 2\check{U}(r,t)+2\check{U}(t,s).
\end{equation*}
{ With the abbreviate notation} $h_k:=h_{\check{T}_k}$, it hence follows
{\begin{align*}
&\int_{[\check{{z}}_{j+\ell-1},\check{{z}}_{j+\ell}]}\int_{[\check{{z}}_{j-1},\check{{z}}_{j+\ell-1}]}{\check{U}(r,s)}\,d{r}\,d{s}\\&\quad=\frac{1}{h_{j+\ell-1}}\int_{[\check{{z}}_{j+\ell-2},\check{{z}}_{j+\ell-1}]}{\int_{[\check{{z}}_{j+\ell-1},\check{{z}}_{j+\ell}]}\int_{[\check{{z}}_{j-1},\check{{z}}_{j+\ell-2}]}{\check{U}(r,s)}\,d{r}\,d{s}}\,d{t}\\
&\quad\leq \frac{2}{h_{j+\ell-1}}\int_{[\check{{z}}_{j+\ell-2},\check{{z}}_{j+\ell-1}]}\int_{[\check{{z}}_{j-1},\check{{z}}_{j+\ell-2}]}{\check{U}(r,t)\int_{[\check{{z}}_{j+\ell-1},\check{{z}}_{j+\ell}]}{1}\,d{s}}\,d{r}\,d{t}\\
&\quad\quad+\frac{2}{h_{j+\ell-1}}\int_{[\check{{z}}_{j+\ell-2},\check{{z}}_{j+\ell-1}]}\int_{[\check{{z}}_{j+\ell-1},\check{{z}}_{j+\ell}]}{\check{U}(t,s)\int_{[\check{{z}}_{j-1},\check{{z}}_{j+\ell-2}]}{1}\,d{r}}\,d{s}\,d{t}\\
&\quad\leq \frac{h_{j+\ell}}{h_{j+\ell-1}}|{u\circ\gamma}|_{H^{1/2}([\check{{z}}_{j-1},\check{{z}}_{j+\ell-1}])}^2+\frac{\check{{z}}_{j+\ell-2}-\check{{z}}_{j-1}}{h_{j+\ell-1}}|{u\circ\gamma}|_{H^{1/2}(\check{T}_{j+\ell-1}\cup \check{T}_{j+\ell})}^2.
\end{align*}}There holds
\begin{align*}\frac{\check{{z}}_{j+\ell-2}-\check{{z}}_{j-1}}{h_{j+\ell-1}}=\sum_{k=j}^{j+\ell-2}\frac{h_k}{h_{j+\ell-1}}\leq \sum_{k=j}^{j+\ell-2}\kappa(\mathcal{T}_h)^{j+\ell-1-k}=\sum_{k=1}^{\ell-1}\kappa(\mathcal{T}_h)^{k}.
\end{align*}
This implies
{\begin{align*}
&\int_{[\check{{z}}_{j+\ell-1},\check{{z}}_{j+\ell}]}\int_{[\check{{z}}_{j-1},\check{{z}}_{j+\ell-1}]}{\check{U}(r,s)}\,d{r}\,d{s}\\&\quad\leq\kappa(\mathcal{T}_h) |{u\circ\gamma}|_{H^{1/2}([\check{{z}}_{j-1},\check{{z}}_{j+\ell-1}])}^2+|{u\circ\gamma}|_{H^{1/2}(\check{T}_{j+\ell-1}\cup \check{T}_{j+\ell})}^2\sum_{k=1}^{\ell-1}\kappa(\mathcal{T}_h)^{k}.
\end{align*}}
Inserting this into \eqref{eq:ind step Hseminorm[1/2] leq} and using 
\begin{equation*}
1+2\sum_{k=1}^{\ell-1}\kappa(\mathcal{T}_h)^{k}\leq(1+2\kappa(\mathcal{T}_h))^{\ell-1}
\end{equation*}
as well as  the induction hypothesis  \eqref{eq:Hsnorm le Hsnorm local hyp}, we obtain
{\begin{align*}
&|{u\circ\gamma}|_{H^{1/2}([\check{{z}}_{j-1},\check{{z}}_{j+\ell}])}^2\\
&\quad \leq (1+2\kappa(\mathcal{T}_h))|u\circ\gamma|_{H^{1/2}([\check{{z}}_{j-1},\check{{z}}_{j+\ell-1}])}^2 +  (1+2\kappa(\mathcal{T}_h))^{\ell-1}|{u\circ\gamma}|_{H^{1/2}(\check{T}_{j+\ell-1}\cup \check{T}_{j+\ell})}^2 \\
&\quad\leq (1+2\kappa(\mathcal{T}_h))^{\ell-1}\sum_{k=j}^{j+\ell-2}|{u\circ\gamma}|_{H^{1/2}(\check{T}_k\cup \check{T}_{k+1})}^2  +  (1+2\kappa(\mathcal{T}_h))^{\ell-1}|{u\circ\gamma}|_{H^{1/2}(\check{T}_{j+\ell-1}\cup \check{T}_{j+\ell})}^2\\
&\quad=(1+2\kappa(\mathcal{T}_h))^{\ell-1}\sum_{k=j}^{j+\ell-1}|{u\circ\gamma}|_{H^{1/2}(\check{T}_k\cup \check{T}_{k+1})}^2.
\end{align*}}This concludes the induction step and thus proves \eqref{eq:Hsnorm le Hsnorm local ass}. 
 There is a $j\in\Z$ with 
  \begin{equation*}
 \gamma([\check{{z}}_{j-1},\check{{z}}_{\min\{j+2m\,,\,j-1+n\}}])=\omega_{h}^{m}(T).
 \end{equation*} 
Because of Assumption {\rm (A1)}, one can choose $I$ such that $I \subseteq  [\check{{z}}_{j-1}, \check{{z}}_{\min\{j+2m\,,\,j-1+n\}} ]$. We use \eqref{eq:Lnorm patch leq Hseminorm[1/2] I} and \eqref{eq:Hsnorm le Hsnorm local ass}  for $\ell=\min\{2m\,,\,n-1\}$  to see
\begin{align*}
&\norm{u}{L^2(\mathrm{supp}(\psi_T))}^2 \leq  \frac{| \mathrm{supp}(\psi_T)|}{2q}(1+2\kappa(\mathcal{T}_h))^{\min\{2m\,,\,n-1\}-1}\sum_{k=j}^{\min\{j+2m\,,\,j-1+n\}-1}|{u\circ\gamma}|_{H^{1/2}(\check{T}_k\cup \check{T}_{k+1})}^2\\
&\quad\leq \frac{|\mathrm{supp}(\psi_T)|}{2q}(1+2\kappa(\mathcal{T}_h))^{2m-1}\sum_{k=j}^{\min\{j+2m\,,\,j-1+n\}-1}|{u\circ\gamma}|_{H^{1/2}(\check{T}_k\cup \check{T}_{k+1})}^2.
\end{align*}
Finally, we use \eqref{eq:equivalent Hsnorm} and
\begin{equation*}
\Big\{{z}_k:k=j,\dots,\min\{j+2m\,,\,j-1+n\}-1\Big\}\subseteq \omega_h^{m-1}(T)\cap \mathcal{N}_h,
\end{equation*}
to get
\begin{align*}
 \sum_{k=j}^{\min\{j+2m\,,\,j-1+n\}-1}|{u\circ\gamma}|_{H^{1/2}(\check{T}_k\cup \check{T}_{k+1})}^2 &\leq \Cgamma^2 \sum_{k=j}^{\min\{j+2m\,,\,j-1+n\}-1}|{u}|_{H^{1/2}(\omega_h({z}_k))}^2\\
&\leq \Cgamma^{2}\sum_{{z} \in \omega_h^{m-1}(T)\cap \mathcal{N}_h}|{u}|_{H^{1/2}(\omega_h({z}))}^2,
\end{align*}
which concludes the proof.
\end{proof}
\begin{proof}[Proof of Lemma \ref{lem:Lnorm le h Hsnorm} for open $\Gamma_{}\subsetneqq \partial\Omega$]
The proof works essentially as before, where  \eqref{eq:Hsnorm le Hsnorm local ass} now becomes
{\begin{align*}
&\forall j\in\N \left( j+\ell\leq n \Longrightarrow  |{u\circ\gamma}|_{H^{1/2}([\check{{z}}_{j-1},\check{{z}}_{j+\ell}])}^2\leq (1+2\kappa(\mathcal{T}_h))^{\ell-1}\sum_{k=j}^{j+\ell-1}|{u\circ\gamma}|_{H^{1/2}(\check{T}_{k}\cup\check{T}_{k+1})}^2\right).
\end{align*}}Details are found in \cite[Lemma 2.15]{diplarbeit}.
\end{proof}

\begin{proposition}\label{prop:reliability}
Suppose the assumptions~{\rm(A1)--(A2)} and
let $u\in H^{1/2}(\Gamma_{})$ satisfy \eqref{eq:orthogonal}.
Then, there exists a constant $C_3>0$ which depends only on $\Gamma_{}$, $m$, 
$\kappa(\TT_h)$, and $q$ such that
\begin{align}\label{eq:reliability}
\norm{u}{H^{1/2}(\Gamma_{})}^2
\le C_3\,\sum_{{z}\in\NN_h}|u|_{H^{1/2}(\omega_h({z}))}^2.
\end{align}
\end{proposition}

\begin{proof}[Proof of Proposition \ref{prop:reliability} for closed $\Gamma_{}=\partial\Omega$]
{Without loss of generality, we may assume that $\gamma=\gamma_L$.} Due to Lemma \ref{lem:Hnorm le Faer plus}, it remains to estimate the term $\sum_{T\in\mathcal{T}_h} h_T^{-1}\norm{u}{L^2(T)}^2$.
For $m=0$, we see
\begin{align*}
C_2^{-1}\sum_{T\in \mathcal{T}_h} h_{T}^{-1} \norm{u}{L^2(T)}^2&\leq \sum_{T\in\mathcal{T}_h} |{u}|_{H^{1/2}(T)}^2\leq \sum_{{z} \in \mathcal{N}_h}|{u}|_{H^{1/2}(\omega_h({z}))}^2.
\end{align*}
For $m>0$, Assumption {\rm (A1)} and Lemma \ref{lem:Lnorm le h Hsnorm} give 
\begin{equation}\label{eq:Lnorm le h semilocal Hsnorm local}
 \norm{u}{L^2(T)}^2\leq\norm{u}{L^2(\mathrm{supp}(\psi_T))}^2\leq C_2 | \omega_h^{m}( T)|\sum_{{z} \in \omega_h^{m-1}( T)\cap \mathcal{N}_h}|{u}|_{H^{1/2}(\omega_h({z}))}^2.
\end{equation}
Let $j\in \{1,\dots,n\}$ with $T=T_j$.
We extend the mesh data periodically.
With {the abbreviate notation} $h_{\ell}:=h_{\check{T}_{\ell}}$, we see
\begin{equation}\label{eq:patchm/hm}
\frac{|\omega_h^{m}(T)|}{h_{T}}\leq\frac{\check{{z}}_{j+m}-\check{{z}}_{j-1-m}}{h_{j}} =\sum_{\ell=-m+1}^{m+1}\frac{h_{j-1+\ell}}{h_{j}}\leq \sum_{\ell=-m+1}^{m+1}\kappa(\mathcal{T}_h)^{|\ell-1|}
\end{equation}
Combining \eqref{eq:Lnorm le h semilocal Hsnorm local} and \eqref{eq:patchm/hm},  we obtain with $C_3:=C_2\sum_{\ell=-m+1}^{m+1}\kappa(\mathcal{T}_h)^{|\ell-1|}$
\begin{align}
\sum_{T\in \mathcal{T}_h} h_{T}^{-1} \norm{u}{L^2(T)}^2&\leq  C_3\sum_{T\in\mathcal{T}_h}\sum_{{z} \in \omega_h^{m-1}( T)\cap \mathcal{N}_h}|{u}|_{H^{1/2}(\omega_h({z}))}^2
=  C_3 \sum_{T\in\mathcal{T}_h}\sum_{\substack{{z} \in \mathcal{N}_h \\ {z} \in \omega_h^{m-1}(T)}}|{u}|_{H^{1/2}(\omega_h({z}))}^2\notag\\
&=C_3 \sum_{{z} \in \mathcal{N}_h}\sum_{\substack{T\in\mathcal{T}_h\\ {z} \in \omega_h^{m-1}(T)}}|{u}|_{H^{1/2}(\omega_h({z}))}^2
\label{eq:last step prop2 sec1} = 2C_3 m\sum_{{z} \in \mathcal{N}_h}|{u}|_{H^{1/2}(\omega_h({z}))}^2.
\end{align}
This concludes the proof.
\end{proof}
\begin{proof}[Proof of Proposition \ref{prop:reliability} for open $\Gamma_{}\subsetneqq\partial\Omega$]
The proof works essentially as for $\Gamma=\partial\Omega$. For details we refer to \cite[Proposition 2.16]{diplarbeit}.
\end{proof}

\begin{proof}[Proof of Theorem~\ref{thm:faermann}, eq.~\eqref{eq2:thm:faermann}]
Galerkin BEM ensures the Galerkin orthogonality
\begin{align*}
 \int_{\Gamma_{}} r_h(x)u_h(x)\,dx= \int_{\Gamma_{}} \big(V(\phi-\phi_h)\big)(x)u_h(x)\,dx = 0
 \quad\text{for all }u_h\in\XX_h
\end{align*}
and hence guarantees~\eqref{eq:orthogonal} for the residual $r_h=f-V\phi_h = V(\phi-\phi_h)$.
Since $V$ is an isomorphism, $\norm{r_h}{H^{1/2}(\Gamma)}\simeq \norm{\phi-\phi_h}{\H^{-1/2}(\Gamma)}$
together with~\eqref{eq:reliability} proves~\eqref{eq2:thm:faermann}.
\end{proof}

\section{Adaptive IGABEM}
\label{section:abem}


\subsection{B-splines and NURBS}
Throughout this subsection, we consider \textit{knots} $\check{\mathcal{K}}:=(t_i)_{i\in\Z}$ on $\R$ with $t_{i-1}\leq t_{i}$ for $i\in \Z$ and $\lim_{i\to \pm\infty}t_i=\pm \infty$.
For the multiplicity of any knot $t_i$, we write $\#t_i$.
We denote the corresponding set of \textit{nodes} $\check{\mathcal{N}}:=\set{t_i}{i\in\Z}=\set{\check{{z}}_{j}}{j\in \Z}$ with $\check{{z}}_{j-1}<\check{{z}}_{j}$ for $j\in\Z$. 
For $i\in\Z$, the $i$-th \textit{B-Spline} of degree $p$ is defined inductively by
\begin{align}
\begin{split}
B_{i,0}&:=\chi_{[t_{i-1},t_{i})},\\
B_{i,p}&:=\beta_{i-1,p} B_{i,p-1}+(1-\beta_{i,p}) B_{i+1,p-1} \quad \text{for } p\in \N,
\end{split}
\end{align}
where, for $t\in \R$,
\begin{align*}
\beta_{i,p}(t):=
\begin{cases}
\frac{t-t_i}{t_{i+p}-t_i} \quad &\text{if } t_i\neq t_{i+p},\\
0 \quad &\text{if } t_i= t_{i+p}.
\end{cases}
\end{align*}
 We also use the notations $B_{i,p}^{\check{\mathcal{K}}}:=B_{i,p}$ and $\beta_{i,p}^{\check{\mathcal{K}}}:=\beta_{i,p}$ to stress the dependence on the knots~$\check{\mathcal{K}}$.
 The proof of the following theorem is found in \cite[Theorem 6]{Boor-SplineBasics}.
\begin{theorem}\label{thm:spline basis}
Let $I=[a,b)$ be a finite interval and $p\in \N_0$.
Then
\begin{equation}
\set{B_{i,p}|_I}{i\in \Z, B_{i,p}|_I\neq 0}
\end{equation}
is a basis for the space of all right-continuous $\check{\mathcal{N}}-$piecewise polynomials of degree lower or equal $p$ on $I$ and which are, at each knot $t_i$, $p-\#t_i$ times continuously differentiable if $p-\#t_i\geq 0$.
\end{theorem}

In addition to the knots $\check{\mathcal{K}}=(t_i)_{i\in\Z}$, we consider positive weights $\mathcal{W}:=(w_i)_{i\in\Z}$ with $w_i>0$.
For $i\in \Z$ and $p\in \N_0$, we define the $i$-th  \textit{non-uniform rational B-Spline} of degree $p$ or shortly \textit{NURBS} as
\begin{equation}
R_{i,p}:=\frac{w_iB_{i,p}}{\sum_{\ell\in\Z}  w_{\ell}B_{\ell,p}}.
\end{equation}
 We also use the notation $R_{i,p}^{\check{\mathcal{K}},\mathcal{W}}:=R_{i,p}$.
Note that the denominator is locally finite and never zero as shown in the following lemma.

\begin{lemma}\label{lem:properties for NURBS}
 For $p\in \N_0$ and $i,\ell \in\Z$, the following assertions hold: 
\begin{enumerate}[{\rm(i)}]
\item \label{item:NURBS rational} $R_{i,p}|_{[t_{\ell-1},t_{\ell})}$ is a rational function with nonzero denominator, which can be extended continuously at $t_{\ell}$.
\item \label{item:NURBS local} $R_{i,p}$ vanishes outside the interval $[t_{i-1},t_{i+p})$. 
It is positive on the open interval $(t_{i-1},t_{i+p})$.
\item  It holds $t_{i-1}=t_{i+p}$ if and only if $R_{i,p}=0$.
\item \label{item:NURBS determined} $B_{i,p}$ is completely determined by the $p+2$ knots $t_{i-1},\dots,t_{i+p}$.
$R_{i,p}$ is completely determined by the $3p+2$ knots $t_{i-p-1},\dots,t_{i+2p}$ and   the $2p+1$ weights $w_{i-p},\dots,w_{i+p}$.
Therefore, we will also use the notation
\begin{equation}\label{eq:Rip}
  R(\cdot|t_{i-p-1},\dots,t_{i+2p},w_{i-p},\dots,w_{i+p}):=R_{i,p}.
\end{equation}
\item\label{item:NURBS partition} The NURBS functions of degree $p$ form a partition of unity, i.e.
\begin{equation}
\sum_{i \in\Z} R_{i,p}=1\quad \text{on }\R.
\end{equation}
\item If all weights are equal, then $R_{i,p}=B_{i,p}$.
Hence, B-splines are just special NURBS functions. 
\item \label{item:B is R} Each NURBS function $R_{i,p}$ is at least $p-\#t_{\ell}$ times continuously differentiable at $t_{\ell}$ if $p-\#t_{\ell}\geq 0$.
\item \label{item:NURBS translated}  For $s,t\in\R$ and $c>0$, we have
\begin{equation}
\forall t\in \R:\quad R_{i,p}^{s+\check{\mathcal{K}},\mathcal{W}}(t)=R_{i,p}^{\check{\mathcal{K}},\mathcal{W}}(t-s)
\end{equation}
as well as
\begin{equation}
\forall t\in \R:\quad R_{i,p}^{c\check{\mathcal{K}},\mathcal{W}}(t)=R_{i,p}^{\check{\mathcal{K}},\mathcal{W}}(t/c).
\end{equation}
 \item \label{item:NURBS convergence} Let $\check{\mathcal{K}}_{\ell}=(t_{i,\ell})_{i\in\Z}$ be a sequence of knots such that $\#t_{i,\ell}=\#t_i$ for all $i\in\Z$ and, $\mathcal{W}_{\ell}=(w_{i,\ell})_{i\in\Z}$ a sequence of positive weights.
If $(\check{\mathcal{K}}_{\ell})_{\ell\in\N}$ converges pointwise to $\check{\mathcal{K}}$ and $(\mathcal{W}_{\ell})_{\ell\in\N}$ converges pointwise to $\mathcal{W}$, then $\big(R_{i,p}^{\check{\mathcal{K}}_{\ell},\mathcal{W}_{\ell}}\big)_{\ell\in\N}$ converges almost everywhere to $R_{i,p}^{\check{\mathcal{K}},\mathcal{W}}$ for all $i\in \N$.

\end{enumerate}
\end{lemma}
\begin{proof}
The proof for (i)--(v) can be found in \cite[Section 2, page 9--10]{Boor-SplineBasics} for B-splines.
The generalization to NURBS is trivial.
(vi) is an immediate consequence of (v).
(vii)~follows from Theorem \ref{thm:spline basis}.
To prove  (viii), we note that
for all $\ell\in\Z$ and $t\in\R$ it holds
\begin{align*}
\chi_{[s+t_{\ell-1},s+t_{\ell})}(t)=\chi_{[t_{\ell-1},t_{\ell})}(t-s)\quad\text{and}\quad\chi_{[ct_{\ell-1},ct_{\ell}+s)}(t)=\chi_{[t_{\ell-1},t_{\ell})}(t/c)
\end{align*}
as well as
\begin{align*}
\frac{t-(s+t_{\ell})}{(s+t_{\ell+p})-(s+t_{\ell})}=\frac{(t-s)-t_{\ell}}{t_{\ell+p}-t_{\ell}}\quad\text{and}\quad\frac{t-ct_{\ell}}{ct_{\ell+p}-ct_{\ell}}=\frac{t/c-t_{\ell}}{t_{\ell+p}-t_{\ell}}.
\end{align*}
Hence, the assertion is an immediate consequence of the definition of B-splines.
For B-splines, (ix) is  proved by induction, noting that for all $p'\in\N$ and $i\in \Z$, we have
\begin{equation*}
\beta_{i,p'}^{\check{\mathcal{K}}_{\ell}}\stackrel{a.e.}{\longrightarrow} \beta_{i,p'}^{\check{\mathcal{K}}}\qquad\text{and}\qquad B_{i,0}^{\check{\mathcal{K}}_{\ell}}\stackrel{a.e.}{\longrightarrow}B_{i,0}^{\check{\mathcal{K}}}.
\end{equation*}
This easily implies the convergence of $R_{i,p}^{\check{\KK}_\ell}$.
\end{proof}

For any $p\in\N_0$, we define the vector spaces
\begin{equation}
\mathscr{S}^p(\check{\mathcal{K}}):=\left\{\sum_{i\in\Z}a_i B_{i,p}:a_i\in\R\right\}
\end{equation}
as well as
\begin{equation}\label{eq:NURBS space defined} 
\mathscr{N}^p(\check{\mathcal{K}},\mathcal{W}):=\left\{\sum_{i\in\Z}a_i R_{i,p}:a_i\in\R\right\}=\frac{\mathscr{S}^p(\check{\mathcal{K}})}{\sum_{i\in \Z} w_i B_{i,p}^{\check{\mathcal{K}}}}.
\end{equation}
 Note that the sums are locally finite.

An analogous version of the following result is already found in \cite{faermann2d} for the special case of B-splines of degrees $p=0,1,2$ and knot multiplicity $\#t_i=1$ for all $i\in \Z$ and weight function $\varphi=1$. 
The following generalization to arbitrary NURBS, however, requires a completely new idea.

\begin{lemma}\label{lem:NURBS satisfy Assumption}
Let $I$ be a compact interval with nonempty interior, $\kappa_{\max}\geq 1$, $0<w_{\min}\leq w_{\max}$ real numbers, $p\in \N_0$, and $\varphi:I\to \R^+$ a piecewise continuously differentiable function   with positive infimum.
Then there exists a constant 
\begin{equation*}
q=q\big(\kappa_{\max},w_{\min},w_{\max},p,\varphi \big)\in (0,1]
\end{equation*} 
such that for arbitrary knots $t_0\leq\dots \leq t_{3p+1}\in I$  and corresponding nodes $\check{z}_0,\dots,\check{z}_m$ with
\begin{align}\label{eq:kappa-knots}
\kappa(t_0,\dots,t_{3p+1}):=\max\left\{\frac{\check{z}_{j+1}-\check{z}_{j}}{\check{z}_{j}-\check{z}_{j-1}},\frac{\check{z}_{j}-\check{z}_{j-1}}{\check{z}_{j+1}-\check{z}_{j}}:j=1,\dots,m-1\right\}\leq \kappa_{\max},
\end{align}
weights   $w_{\min}\leq w_1,\dots,w_{2p+1}\leq w_{\max}$  and all $\ell\in\{p+1,\dots,2p+1\}$,
\begin{equation}\label{eq:my inequality}
\big\|\big(1-R(\cdot|t_0,\dots,t_{3p+1},w_1,\dots,w_{2p+1})\big)\cdot\varphi\big\|_{L^1([t_{\ell-1},t_{\ell}])}\leq (1-q) \norm{\varphi}{L^1([t_{\ell-1},t_{\ell}])}.
\end{equation}
  Note that there holds 
\begin{equation*}
\mathrm{supp}\big(R(\cdot|t_0,\dots,t_{3p+1},w_1,\dots,w_{2p+1})\big)=[t_{p},t_{2p+1}].
\end{equation*}
\end{lemma}

\begin{proof}
We prove the lemma in five steps.

\noindent {\bf Step 1:} We give an abstract formulation of the problem.
For $1\leq \nu \leq 3p+1$, we define the bounded set
 \begin{align*}
M_{\nu}&:=\Big\{(\check{z}_{0},\dots,\check{z}_{\nu},w_{1},\dots,w_{2p+1})\in I^{\nu}\times [w_{\min},w_{\max}]^{2p+1}:\check{z}_{0}<\check{z}_{1}, \\
&\forall m\in\{2,\dots,\nu\}:
\frac{1}{\kappa_{\max}}\big(\check{z}_{m-1}-\check{z}_{m-2}\big)\leq\check{z}_{m}-\check{z}_{m-1}\leq \kappa_{\max}\big(\check{z}_{m-1}-\check{z}_{m-2}\big)\Big\}.
\end{align*}
Note that $(\check{z},w)\in M_{\nu}$ already implies $\check{z}_0<\dots<\check{z}_{\nu}$.
For a vector of multiplicities $k\in \N^{\nu+1}$ with $\sum_{m=0}^{\nu} k_m=3p+2$ we introduce the function
\begin{equation*}
g_{ k,\nu}:\R^{\nu}\to \R^{3p+2}:(\check{z}_{0},\dots,\check{z}_{\nu})\mapsto (\underbrace{\check{z}_{0},\dots,\check{z}_{0}}_{ k_0-\text{times}},\dots,\underbrace{\check{z}_{\nu},\dots,\check{z}_{\nu}}_{ k_{\nu}-\text{times}}).
\end{equation*}
Moreover, we define for $\ell \in \{p+1,\dots,2p+1\}$ the function 
\begin{align*}
\Phi_{ k,\ell,\nu}&:M_{\nu}\to\R:(\check{z},w)\mapsto \frac{\big\|\big(1-R(\cdot|g_{ k,\nu}(\check{z}),w)\big)\cdot\varphi\big\|_{L^1([g_{ k,\nu}(\check{z})_{\ell-1},g_{ k,\nu}(\check{z})_{\ell}])}}{\norm{\varphi}{L^{1}([g_{ k,\nu}(\check{z})_{\ell-1},g_{ k,\nu}(\check{z})_{\ell}])}},
\end{align*}
where $\frac{0}{0}:=0$.
Our aim is to show that for arbitrary $k,\ell,\nu$ there holds $\sup(\Phi_{k,\ell,\nu}(M_{\nu}))<1$.
Then, we define the constant $(1-q)$  as the maximum of all these suprema.
Note that the maximum is taken over a finite set,  since $\sum_{m=0}^{\nu} k_m=3p+2$, $\ell\in \{p+1,\dots,2p+1\}$ and $1\leq \nu\leq 3p+1$. 
  Before we proceed, we show that $(1-q)$ really has the desired properties. 
Without loss of generality, we can assume that not all considered knots $t_{
0},\dots,t_{3p+1}$ are equal.
The corresponding nodes $\check{z}_0,\dots,\check{z}_{\nu}$ and weights $w_1,\dots,w_{2p+1}$ are in $M_{\nu}$. 
If $k$ is the corresponding multiplicity vector, \eqref{eq:my inequality} can indeed be equivalently written as 
\begin{equation*}
\Phi_{k,\ell,\nu}(\check{z},w)\leq(1-q).
\end{equation*}
{\bf Step 2:} \label{item:Phi continuous and lower 1} We fix $k, \ell,\nu$. 
Without loss of generality, we assume that there exists $0\leq \widetilde{\nu}\leq \nu$ such that $\ell-1=\sum_{m=0}^{\widetilde{\nu}}k_m$. 
This just means that the appearing integrals have nonempty integration domains $[g_{ k,\nu}(\check{z})_{\ell-1},g_{ k,\nu}(\check{z})_{\ell}]$, since in this case $\Phi_{k,\ell,\nu}(\check{z},w)=0$ is already bounded.
Using Lemma \ref{lem:properties for NURBS}, (\ref{item:NURBS local}) and (\ref{item:NURBS partition}), we see  that for $(\check{z},w)\in M_{\nu}$, the function $R(\cdot|g_{ k,\nu}(\check{z}),w)$ attains only values in $[0,1]$ and is positive on the interval $\big(g_{k,\nu}(\check{z})_{\ell-1},g_{k,\nu}(\check{z})_{\ell}\big)$.
This implies 
\begin{equation}\label{eq:ran(Phi)<1}
\Phi_{ k,\ell,\nu}(M_{\nu})\subseteq[0,1).
\end{equation}
Because of Lemma \ref{lem:properties for NURBS}, (\ref{item:NURBS convergence}), we can apply Lebesgue's dominated convergence theorem to see that $\Phi_{ k,\ell,\nu}$ is continuous. {If $M_{\nu}$ was compact, we would be done. Unfortunately it is not.}

\noindent{\bf Step 3:} \label{item:varphi=1} Now, we prove the lemma for $\varphi=1$. 
In the definition of $M_{\nu}$ we replace the interval $I$ by $\R$ to define a superset of $M_{\nu}$
\begin{align*}
M_{\nu,\R}&:=\Big\{(\check{z},w)\in \R^{\nu}\times [w_{\min},w_{\max}]^{2p+1}:\\
&\check{z}_{0}<\check{z}_{1}, 
\forall m\in\{2,\dots,\nu\}:\\
&\frac{1}{\kappa_{\max}}\big(\check{z}_{m-1}-\check{z}_{m-2}\big)
\leq\check{z}_{m}-\check{z}_{m-1}\leq \kappa_{\max}\big(\check{z}_{m-1}-\check{z}_{m-2}\big)\Big\}.
\end{align*}
We extend the function $\Phi_{k,\ell,\nu}$ to
\begin{align*}
\widetilde{\Phi}_{ k,\ell,\nu}&:M_{\nu,\R}\to\R:(\check{z},w)\mapsto \frac{\big\|1-R(\cdot|g_{ k,\nu}(\check{z}),w)\big\|_{L^1([g_{ k,\nu}(\check{z})_{\ell-1},g_{ k,\nu}(\check{z})_{\ell}])}}{g_{ k,\nu}(\check{z})_{\ell}-g_{ k,\nu}(\check{z})_{\ell-1}}.
\end{align*}
We define a  closed and bounded and hence  compact subset of $M_{\nu}$
\begin{align*}
& M_{\nu,\R}^{0,1}:=\set{(\check{z},w)\in M_{\nu,\R}}{\check{z}_0=0, \check{z}_1=1}.
\end{align*}
 If $(\check{z},w)\in M_{\nu,\R}$, then
$\big(\frac{\check{z}-\check{z}_0}{\check{z}_1-\check{z}_0},w\big)\in M_{\nu,\R}^{0,1}$
and due to the substitution rule and Lemma~\ref{lem:properties for NURBS}, (\ref{item:NURBS translated}), there holds with the notation 
$\fint_c^d (\cdot)(t) \, d t=\int_c^d (\cdot)(t) \, d t/(d-c)$
 \begin{align*}
\widetilde{\Phi}_{k,\ell,\nu}(\check{z},w)&=\fint_{g_{k,\nu}(\check{z})_{\ell-1}}^{g_{k,\nu}(\check{z})_{\ell}}\big(1-R(t|g_{k,\nu}(\check{z}),w)\big)  \, d t\\
&=\fint_{\frac{g_{k,\nu}(\check{z})_{\ell-1}-\check{z}_0}{\check{z}_1-\check{z}_0}}^{\frac{g_{k,\nu}(\check{z})_{\ell}-\check{z}_0}{\check{z}_1-\check{z}_0}}\big(1-R\big(t(\check{z}_1-\check{z}_0)+\check{z}_0|g_{k,\nu}(\check{z}),w\big)\big) \, d t\\
&=\widetilde{\Phi}_{k,\ell,\nu}\left(\frac{\check{z}-\check{z}_0}{\check{z}_1-\check{z}_0}, w\right).
\end{align*}
Hence we have 
\begin{equation*}
\widetilde{\Phi}_{ k,\ell,\nu}(M_{\nu,\R})=\widetilde{\Phi}_{ k,\ell,\nu}(M_{\nu,\R}^{0,1}).
\end{equation*}
As in \eqref{item:Phi continuous and lower 1} one sees that $\widetilde{\Phi}_{ k,\ell,\nu}$ only attains values in $[0,1)$ and is continuous. 
Since $M_{\nu,\R}^{0,1}$ is compact we get 
\begin{equation*}
\sup\big(\Phi_{ k,\ell,\nu}(M_{\nu})\big)\leq\sup\big(\widetilde{\Phi}_{ k,\ell,\nu}(M_{\nu,\R})\big)<1. 
\end{equation*}
This proves the lemma for $\varphi=1$.
\item \label{item:varphi=estimators} We prove the lemma for $\varphi=c_1\chi_{(-\infty,T)}|_I+c_2\chi_{[T,\infty)}|_I$ with $c_1,c_2>0$ and $T\in I$.
Again, we extend the function $\Phi_{k,\ell,\nu}$ to $M_{\nu,\R}$ 
\begin{align*}
\widetilde{\Phi}_{ k,\ell,\nu}&:M_{\nu,\R}\to\R:
(\check{z},w)\mapsto\\
& \frac{\big\|\big(1-R(\cdot|g_{ k,\nu}(\check{z}),w)\big)\big(c_1\chi_{(-\infty,T)}+c_2\chi_{[T,\infty)}\big)\big\|_{L^1([g_{ k,\nu}(\check{z})_{\ell-1},g_{ k,\nu}(\check{z})_{\ell}])}}{\norm{c_1\chi_{(-\infty,T)}+c_2\chi_{[T,\infty)}}{L^1([g_{ k,\nu}(\check{z})_{\ell-1},g_{ k,\nu}(\check{z})_{\ell}])}}.
\end{align*}
For the proof of the lemma, it is sufficient to show $\sup\big(\widetilde{\Phi}_{ k,\ell,\nu}(M_{\nu,\R})\big)<1$.
Due to the substitution rule and Lemma \ref{lem:properties for NURBS}, (\ref{item:NURBS translated}), we can assume without loss of generality that $T=0$.
Because of \eqref{item:varphi=1} it only remains to show that
\begin{equation*}
\sup\big(\widetilde{\Phi}_{ k,\ell,\nu}(\set{(\check{z},w)\in M_{\nu,\R}}{\check{z}_0\leq 0\leq \check{z}_{\nu}}\big)<1.
\end{equation*}
As in \eqref{item:Phi continuous and lower 1}, one verifies that $\widetilde{\Phi}_{k,\ell,\nu}$ only attains values in $[0,1)$ and is continuous.
Moreover, due to the substitution rule and Lemma \ref{lem:properties for NURBS}, (\ref{item:NURBS translated}), we have for any element of $\set{(\check{z},w)\in M_{\nu,\R}}{\check{z}_0\leq 0\leq \check{z}_{\nu}}$
 \begin{align*}
\widetilde{\Phi}_{k,\ell,\nu}(\check{z},w)
=\widetilde{\Phi}_{k,\ell,\nu}\left(\frac{\check{z}}{\check{z}_1-\check{z}_0},w\right)
\end{align*}
and hence
\begin{align*}
&\widetilde{\Phi}_{ k,\ell,\nu}(\set{(\check{z},w)\in M_{\nu,\R}}{\check{z}_0\leq 0\leq \check{z}_{\nu}}\big)\\
&\quad=\widetilde{\Phi}_{ k,\ell,\nu}(\set{(\check{z},w)\in M_{\nu,\R}}{\check{z}_1-\check{z}_0=1,\check{z}_0\leq 0\leq \check{z}_{\nu}}\big).
\end{align*}
The second set is compact, since it is the image of a closed and bounded set under a continuous mapping.
Therefore it attains a maximum smaller than one.
This concludes the proof for $\varphi=c_1\chi_{(-\infty,T)}|_I+c_2\chi_{[T,\infty)}|_I$.

\noindent{\bf Step 4:} Finally, we are in the position to prove the assertion of the lemma for arbitrary functions $\varphi$ with the desired properties.
Let $\big((\check{z}^m,w^m)\big)_{m\in \N}$ be a sequence in $M_{\nu}$ such that the $\Phi_{k,\ell,\nu}$-values converge to $\sup(\Phi_{k,\ell,\nu}(M_{\nu}))$. 
Because of the boundedness of $M_{\nu}$, we can assume convergence of the sequence, where the limit $(\check{z}^{\infty},w^{\infty})$ is in $\overline{M_{\nu}}$, i.e. $(\check{z}^{\infty},w^{\infty})\in M_{\nu}$ or $(\check{z}^{\infty},w^{\infty})\in I^{\nu}\times [w_{\min},w_{\max}]^{2p+1}$ with $\check{z}_0^{\infty}=\dots=\check{z}_{\nu}^{\infty}$.
In the first case, we are done because of \eqref{eq:ran(Phi)<1} and the continuity of $\Phi_{k,\ell,\nu}$.
For the second case, we define 
\begin{equation*} 
a_n:=g_{k,\nu}(\check{z}^n,w^n)_{\ell-1},b_n:=g_{k,\nu}(\check{z}^n,w^n)_{\ell} \quad \text{and}\quad R_n:=R(\cdot|\check{z}^n,w^n).
\end{equation*}
Note that $a_n<b_n$, and that the sequences $(a_n)_{n\in\N}$ and $(b_n)_{n\in\N}$ converge to the limit 
\begin{equation*}
Z:=\check{z}_0^{\infty}=\dots=\check{z}_{\nu}^{\infty}\in I.
\end{equation*}
We consider two cases.

Case 1: If $\varphi$ is continuous at the limit $Z$, it is absolutely continuous on the interval $[a_n,b_n]$ for sufficiently big $n\in \N$.
Hence we have for sufficiently big $n\in \N$
\begin{align*}
\Phi_{k,\ell,\nu}(\check{z}^n,w^n)&=\frac{\int_{a_n}^{b_n}{\big(1-R_n(t)\big)\varphi(t)}\, d t}{\int_{a_n}^{b_n}{\varphi(t)}\, d t}\\
&=\frac{\int_{a_n}^{b_n}{\big(1-R_n(t)\big)\big(\varphi(a_n)+\int_{a_n}^{t}{\varphi'( \tau)}\, d  \tau\big)}\, d t}{\int_{a_n}^{b_n}{\big(\varphi(a_n)+\int_{a_n}^{t}{\varphi'( \tau)}\, d  \tau\big)}\, d t}\\
&\leq \frac{\int_{a_n}^{b_n}{\big(1-R_n(t)\big)\varphi(a_n)}\, d t+(b_n-a_n)^2\norm{\varphi'}{L^\infty(I)}}{(b_n-a_n)\varphi(a_n)-(b_n-a_n)^2\norm{\varphi'}{L^\infty(I)}}.
\end{align*}
The second summand converges to zero.
We consider the first one.
For any $C\in (0,1)$, there holds for sufficiently big $n\in \N$
\begin{align*}
\frac{\int_{a_n}^{b_n}{\big(1-R_n(t)\big)\varphi(a_n)}\, d t}{(b_n-a_n)\varphi(a_n)-(b_n-a_n)^2\norm{\varphi'}{L^\infty(I)}}&\leq \frac{\int_{a_n}^{b_n}{\big(1-R_n(t)\big)\varphi(a_n)}\, d t}{(b_n-a_n)\varphi(a_n)\cdot C}\\
&\leq \frac{1}{C}\Big(1-q\big(\kappa_{\max},w_{\min},w_{\max},p,1 \big)\Big).
\end{align*}
Since $C$ was arbitrary, this implies 
\begin{equation*}
\sup\big(\Phi_{k,\ell,\nu}(M_{\nu})\big)\leq \Big(1-q\big(\kappa_{\max},w_{\min},w_{\max},p,1 \big) \Big)<1.
\end{equation*}

Case 2: If $\varphi$ is not continuous at the limit $Z$ we proceed as follows.
For sufficiently big $n\in \N$, $\varphi$ is absolutely continuous on $[a_n,Z]$ and on $[Z,b_n]$.
By considering suitable subsequences, we can assume that $a_n<b_n\leq Z$, $Z\leq a_n<b_n$ or $a_n\leq Z \leq b_n$, each for all $n\in \N$.
In the first two cases, we can proceed as in Case 1.
In the third case, we argue similarly as in Case 1 to see, with the left-handed limit $\varphi^{\ell}(Z)$ and the right-handed limit $\varphi^{r}(Z)$ for $n\in \N$ big enough 
\begin{align*}
&\Phi_{k,\ell,\nu}(\check{z}^n,w^n)=\frac{\int_{a_n}^{b_n}{\big(1-R_n(t)\big)\varphi(t)}\, d t}{\int_{a_n}^{b_n}{\varphi(t)}\, d t}\\
&\quad=\frac{\int_{a_n}^{Z}{\big(1-R_n(t)\big)\big(\varphi^{\ell}(Z)-\int_{t}^{Z}{\varphi'( \tau)}\, d  \tau\big)}\, d t}{\int_{a_n}^{b_n}{\varphi(t)}\, d t}\\
&\quad\quad +\frac{\int_{Z}^{b_n}{\big(1-R_n(t)\big)\big(\varphi^{r}(Z)+\int_{Z}^{t}{\varphi'( \tau)}\, d  \tau\big)}\, d t}{\int_{a_n}^{b_n}{\varphi(t)}\, d t}\\
&\quad\leq \frac{\int_{a_n}^{b_n}{\big(1-R_n(t)\big)\big(\varphi^{\ell}(Z)\chi_{(-\infty,Z)}(t)+\varphi^r(Z)\chi_{[Z,\infty)}(t)\big)}\, d t}{\int_{a_n}^{b_n}{\varphi^{\ell}(Z)\chi_{(-\infty,Z)}(t)+\varphi^r(Z)\chi_{[Z,\infty)}(t)}\, d t-2(b_n-a_n)^2\norm{\varphi'}{L^\infty(I)}}\\
&\quad\quad+\frac{2(b_n-a_n)^2\norm{\varphi'}{L^\infty(I)}}{\int_{a_n}^{b_n}{\varphi^{\ell}(Z)\chi_{(-\infty,Z)}(t)+\varphi^r(Z)\chi_{[Z,\infty)}(t)}\, d t-2(b_n-a_n)^2\norm{\varphi'}{L^\infty(I)}}.
\end{align*}
Again, the second summand converges to zero, wherefore it remains to consider the first one.
For any $C\in (0,1)$, there holds for sufficiently big $n\in \N$ due to \eqref{item:varphi=estimators}
\begin{align*}
&\frac{\int_{a_n}^{b_n}{\big(1-R_n(t)\big)\big(\varphi^{\ell}(Z)\chi_{(-\infty,Z)}(t)+\varphi^r(Z)\chi_{[Z,\infty)}(t)\big)}\, d t}{\int_{a_n}^{b_n}{\varphi^{\ell}(Z)\chi_{(-\infty,Z)}(t)+\varphi^r(Z)\chi_{[Z,\infty)}(t)}\, d t-2(b_n-a_n)^2\norm{\varphi'}{L^\infty(I)}}\\
&\quad\leq\frac{\int_{a_n}^{b_n}{\big(1-R_n(t)\big)\big(\varphi^{\ell}(Z)\chi_{(-\infty,Z)}(t)+\varphi^r(Z)\chi_{[Z,\infty)}(t)\big)}\, d t}{\int_{a_n}^{b_n}{\varphi^{\ell}(Z)\chi_{(-\infty,Z)}(t)+\varphi^r(Z)\chi_{[Z,\infty)}(t)}\, d t\cdot{C}}\\
&\quad\leq \frac{1}{C}\Big(1-q\big(\kappa_{\max},w_{\min},w_{\max},p,\varphi^{\ell}(Z)\chi_{(-\infty,Z)}|_I+\varphi^r(Z)\chi_{[Z,\infty)}|_I \big)\Big)
\end{align*}
Since $C$ was arbitrary, this implies 
\begin{align*}
&\sup\big(\Phi_{k,\ell,\nu}(M_{\nu})\big)\\
&\quad\leq \Big(1-q\big(\kappa_{\max},w_{\min},w_{\max},p,\varphi^{\ell}(Z)\chi_{(-\infty,Z)}|_I+\varphi^r(Z)\chi_{[Z,\infty)}|_I  \big) \Big)<1,
\end{align*}
which concludes the proof.
\end{proof}

We return to our problem \eqref{eq:strong}. 
If $\Gamma=\partial\Omega$ is closed, each node $\check{z}\in \check{\mathcal{N}}_h$ may be assigned with a multiplicity $\#\check{z}\leq p+1$. 
This induces a sequence of non decreasing knots $\check{\mathcal{K}}_h=(t_i)_{i=1}^N$ on $(a,b]$.
Let $\mathcal{W}_h=(w_{i})_{i=1}^N$ be a sequence of weights on these knots. 
We extend the knot sequence $(b-a)$-periodically to $(t_i)_{i\in \Z}$ and the weight sequence to $(w_i)_{i\in \Z}$ by $w_{n+i}:=w_i$ for $i\in \Z$.
For the extended sequences we also write $\check{\mathcal{K}}_h$ and $\mathcal{W}_h$. 
We set
\begin{equation}
\hat{\mathscr{N}}^p(\check{\mathcal{K}}_h,\mathcal{W}_h):=\mathscr{N}^p(\check{\mathcal{K}}_{h},\mathcal{W}_{h})|_{[a,b)}\circ \gamma|_{[a,b)}^{-1}.
\end{equation}
If $\Gamma_{}\neq \partial\Omega$ is open, we assign to each node $\check{z}\in \check{\mathcal{N}}_h$ a corresponding multiplicity $\#\check{z}\leq p+1$ such that $\#\check{z}_0=\#\check{z}_n=p+1$.
This induces a sequence of non decreasing knots $\check{\mathcal{K}}_h=(t_i)_{i=0}^N$ on $[a,b]$.
Let $\mathcal{W}_h=(w_{i})_{i=1}^{N-p}$ be a sequence of weights.
To keep the notation simple, we extend the sequences arbitrarily to $\check{\mathcal{K}}_h=(t_i)_{i\in\Z}$ with $t_i\leq t_{i+1}$ for $i\in\Z$, $a>t_{i}\to-\infty$ for $i<0$ and $b<t_i\to \infty$ for $i>N$, and $\mathcal{W}_h=(w_i)_{i\in \Z}$ with $w_i>0$. 
This allows to define
\begin{equation}
\hat{\mathscr{N}}^p(\check{\mathcal{K}}_h,\mathcal{W}_h):=\mathscr{N}^p(\check{\mathcal{K}}_h,\mathcal{W}_h)|_{[a,b]}\circ \gamma^{-1}.
\end{equation}
Due to Lemma \ref{lem:properties for NURBS}, \eqref{item:NURBS local} and \eqref{item:NURBS determined}, this definition does not depend on how the sequences are extended.

 With the following theorem we conclude that Theorem \ref{thm:faermann}  holds for the span of transformed NURBS functions. 

\begin{theorem}\label{thm:NURBS satisfy Assumption}
 Let $p\in \N_0$ and $m:=\lceil p/2\rceil$.   
Then, the space   $\XX_h:=\hat{\mathscr{N}}^p(\check{\mathcal{K}}_{h},\mathcal{W}_{h})$  
is a subspace of $L^2(\Gamma_{})$ which satisfies the assumptions {\rm (A1)--(A2)} from Section \ref{sec:main theorem} with the constant of Lemma \ref{lem:NURBS satisfy Assumption}
\begin{equation*}
q=q\big(\kappa( \check{\mathcal{T}}_h),\min(\mathcal{W}_h),\max(\mathcal{W}_h),p,\varphi\big),
\end{equation*}
where $\varphi=|\gamma'|_I|$   with $I=[a-(b-a)(m+p),b+(b-a)(2p-m)]$ resp. $I=[a,b]$.
\end{theorem}
\begin{proof}[Proof of Theorem \ref{thm:NURBS satisfy Assumption} for closed $\Gamma_{}=\partial\Omega$]
Lemma \ref{lem:properties for NURBS}, (\ref{item:NURBS rational}) and (\ref{item:NURBS local}), implies $\mathscr{N}^p(\check{\mathcal{K}}_{h},\mathcal{W}_{h})\leq L^2(\R)$.
This shows
$\hat{\mathscr{N}}^p(\check{\mathcal{K}}_{h},\mathcal{W}_{h})\leq L^2(\Gamma_{}).$

Let $T$ be an element of the mesh $\mathcal{T}_h$, $j\in\{1,\dots,n\}$ with $T=T_j$, and $i\in\{1,\dots,N\}$ with $\check{z}_{j-1}=t_{i-1}$ and $\check{z}_j=t_{i}$.
We define $\check{\psi}_T(t):=R_{i-m,p}(t)$ for $t\in [a,b)$   and extend it continuously at $b$.
 We set $\psi_{T}:=\check{\psi}_T|_{[a,b)}\circ\gamma|_{[a,b)}^{-1}$.
Because of Lemma \ref{lem:properties for NURBS}, (\ref{item:NURBS local}), there holds 
\begin{equation}\label{eq:Assumption parametric}
 \check{T}_j\subseteq \left[t_{i-m-1},t_{i-m+p}\right]\cap[a,b]=\mathrm{supp}(\check{\psi}_T)\subseteq \left[\check{z}_{j-m-1},\check{z}_{j-m+p}\right]\subseteq\left[\check{z}_{j-m-1},\check{z}_{j+m}\right].
\end{equation}
Since   $\gamma|_{[a,a+(b-a)/2]}$ and $\gamma|_{[a+(b-a)/2,b]}$ are homeomorphisms, there holds 
\begin{align}\label{eq:gamma(supp(R_T))}
\begin{split}
&\gamma(\mathrm{supp}(\check{\psi}_T))=\gamma\left(\overline{\set{t\in [a,b)}{\check{\psi}_T(t)\neq 0}}\right)= \mathrm{supp}(\psi_{T}),
\end{split}
\end{align}
wherefore $\mathrm{supp}(\psi_{T})$ is connected.
With \eqref{eq:Assumption parametric}, this shows
\begin{equation*}
T\subseteq \mathrm{supp}(\psi_{T}) \subseteq \omega_h^m(T),
\end{equation*}
and hence implies Assumption {\rm (A1)}.

To verify  Assumption {\rm (A2)}, we apply Lemma~\ref{lem:NURBS satisfy Assumption}.
Note that $R_{i-m,p}$ is completely determined by the knots in $I$ and their weights.
This is due to $I\supseteq [t_{i-m-p-1},t_{i+2p-m}]$ and Lemma \ref{lem:properties for NURBS}, \eqref{item:NURBS determined}.
The regularity constant of these knots from \eqref{eq:kappa-knots} is obviously smaller or equal than $\kappa(\check{\mathcal{K}}_h)$. 
Since $\gamma$ is piecewise two times continuously differentiable and its left and right derivative vanishes nowhere, $|\gamma'|$ is piecewise continuously differentiable and is bounded from above by some positive constant.
With Lemma \ref{lem:NURBS satisfy Assumption} and \eqref{eq:gamma(supp(R_T))}, we hence get
 \begin{align*}
&\|1-\psi_{T}\|_{L^2(\mathrm{supp}(\psi_{T}))}^2=\int_{\mathrm{supp}\left(\check{\psi}_T\right)}{(1-\check{\psi}_T)^2|\gamma'(t)|}\,d{t}\leq \int_{\mathrm{supp}\left(\check{\psi}_T\right)}{(1-\check{\psi}_T)|\gamma'(t)|}\,d{t}\\
&\quad=\|(1-\check{\psi}_T)\varphi\|_{L^1([t_{i-m-1},t_{i-m+p}]\cap[a,b])}\leq (1-q) \|\varphi\|_{L^1([t_{i-m-1},t_{i-m+p}]\cap[a,b])}\\ 
&\quad= (1-q) \int_{\mathrm{supp}(\check{\psi}_T)}{|\gamma'(t)|}\,d{t}= (1-q) |\mathrm{supp}(\psi_T)|.
\end{align*}
Consequently, Assumption {\rm (A2)} is also fulfilled.
This concludes the proof.
\end{proof}

\begin{proof}[Proof of Theorem \ref{thm:NURBS satisfy Assumption} for open $\Gamma_{}\subsetneqq \partial\Omega$]
The proof works analogously as before. 
Details are found in \cite[Theorem 4.14]{diplarbeit}.
\end{proof}

\subsection{Knot insertion}
Before we formulate an adaptive algorithm based on NURBS, we recall refinement by \textit{knot-insertion},  see e.g. \cite[Section 11]{Boor-SplineBasics}. 
For general knots $\check{\KK}=(t_i)_{i\in\Z}$ as in the previous subsection, a polynomial degree $p\in \N_0$, and a refined sequence $\check{\KK}'=(t_i')_{i\in\Z}$  (i.e., $\check{\KK}$ is a subsequence of $\check{\KK}$) Theorem \ref{thm:spline basis} implies nestedness
\begin{align}\label{eq:splines nested}
\mathscr{S}^p(\check{\KK}) \subseteq \mathscr{S}^p(\check{\KK}').
\end{align}
We assume that the multiplicities of the knots in $\check{\KK}'$ are lower or equal $p+1$.
Because of  Lemma \ref{lem:properties for NURBS}, \eqref{item:NURBS local}, and Theorem \ref{thm:spline basis} each element $\sum_{i\in \Z} a_i B_{i,p}^{\check{\KK}}\in \mathscr{S}(\check{\KK})$ admits some unique coefficient vector $(a_i')_{i \in \Z}$ with
\begin{align}
\sum_{i\in \Z} a_i B_{i,p}^{\check{\KK}}=\sum_{i\in \Z} a_i' B_{i,p}^{\check{\KK}'}.
\end{align}
If $\check{\KK}'$ contains only one additional knot $t'$ (possibly already contained in $\check{\mathcal{K}}$), the coefficients can be calculated explicitly.
We assume $t_i=t_i'$ for all $i$ with $t_i<t'$.
Then, \cite[Algorithm 11]{Boor-SplineBasics} shows
\begin{align}\label{eq:new coeffs}
\begin{split}
a_i'=\begin{cases}
a_i \quad &\text{if } t_{i+p}\leq t',\\
(1-\beta_{i-1,p}^{\check{\mathcal{K}}}(t')) a_{i-1} + \beta_{i-1,p}^{\check{\mathcal{K}}}(t') a_i \quad &\text{if } t_i<t'<t_{i+p},\\
a_{i-1}\quad &\text{if } t' \leq t_i.
\end{cases}
\end{split}
\end{align}

For closed $\Gamma=\partial\Omega$, we consider again knots $\check{\KK}_h=(t_i)_{i=1}^N$ and weights $\mathcal{W}_h=(w_i)_{i=1}^N$ as in the previous subsection.
We additionally assume $p+1\leq N$.
Now we insert an additional knot $t'\in (a,b]$ to the knots $\check{\KK}_h$ such that the multiplicities of the new knots $\check{\KK}_h'$ are still smaller or equal than $p+1$.
The new knots are extended $(b-a)$-periodically.
We want to find the unique weights $(w_i')_{i\in \Z}$ which fulfill 
\begin{align}\label{eq:equal weight fct}
\sum_{i\in\Z} w_i B_{i,p}^{\check{\KK}_h}= \sum_{i\in \Z} w_i' B_{i,p}^{\check{\KK}_h'}.
\end{align}
They are obviously $(N+1)$-periodic.
We cannot immediately apply \eqref{eq:new coeffs}, since infinitely many knots $\{t'+k(b-a):k\in\Z\}$ are added to $\check{\KK}_h$.
First, we add $\big\{t'+k(b-a):k\in\Z\setminus\{-1,0,1\}\big\}$ to $\check{\KK}_h$ and obtain $\check{\KK}^+=(t_i^+)_{i\in \Z}$ with $t_0=t_0^+$ and $t_1=t_1^+$.
There exist unique weights $(w_i^+)_{i\in \Z}$ with 
\begin{align*}
\sum_{i\in\Z} w_i B_{i,p}^{\check{\KK}_h}= \sum_{i\in \Z} w_i^+ B_{i,p}^{\check{\KK}^+}.
\end{align*}
With $I:=[t_{-1},t_{N+1})$, Lemma \ref{lem:properties for NURBS}, \eqref{item:NURBS local} and \eqref{item:NURBS determined}, and our assumption $p+1\leq N$ imply
\begin{equation*}
\sum_{i=-p}^{N+1} w_i B_{i,p}^{\check{\mathcal{K}}_h}|_I=\sum_{i=-p}^{N +1}w_i^+ B_{i,p}^{\check{\mathcal{K}}^+}|_I=\sum_{i=-p}^{N+1} w_i^+ B_{i,p}^{\check{\mathcal{K}}_h}|_I.
\end{equation*}
With $t_{N}<t_{N+1}$, it is easy to check that $B_{i,p}^{\check{\mathcal{K}}_h}|_I\neq 0$ for $i=0,\dots,N$.
Hence, Theorem~\ref{thm:spline basis} implies $w_i=w_i^+$ for $i=0,\dots,N$. 
It just remains to add the knots $t'-(b-a)$, $t'$ and $t'+(b-a)$. 
To this end, we can repetitively apply \eqref{eq:new coeffs} to obtain the weights $(w_i')_{i=1}^{N+1}$. 
Note that this only involves the weights $(w_i^+)_{i=0}^N$ are needed.
Moreover, the new weights $(w_i')_{i=1}^{N+1}$ are just convex combinations of the old ones $(w_i)_{i=1}^N$.
With \eqref{eq:NURBS space defined}, \eqref{eq:splines nested}, and \eqref{eq:equal weight fct}, we get nestedness
\begin{align}\label{eq:NURBS nested}
\hat{\mathscr{N}}^p(\check{\KK}_h,\mathcal{W}_h)\subseteq \hat{\mathscr{N}}^p(\check{\KK}_h',\mathcal{W}_h').
\end{align}
For closed $\Gamma\subsetneqq \partial\Omega$, a knot $t'\in (a,b]$ can  analogously be inserted to the knots $\check{\KK}_h=(t_i)_{i=0}^N$.
\subsection{Adaptive algorithm}\label{subsec:algorithm}
In this subsection, we introduce an adaptive algorithm, which uses the local contributions of $\eta_h$ to steer the $h$-refinement of the mesh $\mathcal{T}_h$ as well as the increase of the multiplicity of the nodes $\mathcal{N}_h$. 
To respect the iterative character of this procedure, all discrete quantities (as, e.g., $\mathcal{T}_h$, $\phi_h$, etc.) are indexed by the level $\ell \in \N_0$ of the adaptive process instead of the mesh-size $h$.
Let $0<\theta< 1$ be an adaptivity parameter and $p\in\N_0$ a polynomial degree.
We start with some nodes $\check{\mathcal{N}}_0$.
Each node has a multiplicity lower or equal $p+1$, where for open $\Gamma\subsetneqq \partial\Omega$ we assume $\#a=\#b=p+1$.
This induces knots $\check{\mathcal{K}}_0$.
Let  $\mathcal{W}_0$ be some initial positive weights.
We assume $p+1\leq N_0$ and for closed $\Gamma=\partial\Omega$, $|T|\leq |\Gamma|/4$ for all $T\in \TT_0$.
As the initial trial space, we consider 
\begin{equation}
\XX_0:=\hat{\mathscr{N}}^p(\check{\mathcal{K}}_0,\mathcal{W}_0)\subseteq L^2(\Gamma_{})\subseteq{H}^{-1/2}(\Gamma_{}).
\end{equation}
The adaptive algorithm with \textit{D\"orfler marking} reads as follows:

\begin{algorithm}\label{the algorithm}
\textbf{Input:} Adaptivity parameter $0<\theta<1$, polynomial order $p\in \N_0$, initial mesh $\TT_0$ with knots $\check{\KK}_0$, initial weights $\mathcal{W}_0$.\\
\textbf{Adaptive loop:} Iterate the following steps, until $\eta_\ell$ is sufficiently small:
\begin{itemize}
\item[(i)] Compute discrete solution $\phi_\ell\in\XX_\ell$.
\item[(ii)] Compute indicators $\eta_\ell({z})$
for all nodes ${z}\in\NN_\ell$.
\item[(iii)] Determine a minimal set of nodes $\MM_\ell\subseteq\NN_\ell$ such that
\begin{align}
 \theta\,\eta_h^2 \le \sum_{{z}\in\MM_\ell}\eta_\ell({z})^2.
\end{align}
\item[(iv)] If both nodes of an element $T\in\TT_\ell$ belong to $\mathcal{M}_\ell$, $T$  will be marked.
\item[(v)] For all other nodes in $\mathcal{M}_\ell$, the multiplicity will be increased if it is less or equal to $p+1$, otherwise the elements which contain one of these nodes $z\in\mathcal{M}_\ell$, will be marked.
\item[(vi)] Refine all marked elements $T\in\TT_\ell$ by bisection of the corresponding $\check{T}\in \check{\TT}_\ell$.
Use further bisections to guarantee that the new mesh $\TT_{\ell+1}$ satisfies
\begin{align}\label{eq:kappa small}
\kappa(\check{\TT}_{\ell+1})\leq 2\kappa(\check{\TT}_0).
\end{align}
Update counter $\ell\mapsto \ell+1$.  
\end{itemize}
\textbf{Output:} Approximate solutions $\phi_\ell$ and error estimators $\eta_\ell$ for all $\ell \in \N_0$.
\end{algorithm}
An optimal 1D bisection algorithm which ensures \eqref{eq:kappa small}, is discussed and analyzed in \cite{zbMATH06270343}. Note that boundedness of $\kappa(\check{\TT}_\ell)$ implies as well boundedness of $\kappa(\TT_\ell)$. 
Moreover, there holds 
\begin{align}
\min(\mathcal{W}_0)\leq \min(\mathcal{W}_\ell)\leq \max(\mathcal{W}_{\ell})\leq \max(\mathcal{W}_0),
\end{align}
since the new weights are convex combinations of the old weights.
Hence, Theorem~\ref{thm:faermann} and Theorem~\ref{thm:NURBS satisfy Assumption} apply and show efficiency and reliability of the estimator
\begin{align}
 \Crel^{-1}\,\norm{\phi-\phi_\ell}{\H^{-1/2}(\Gamma_{})}
 \le \eta_\ell
 \le \Ceff\,\norm{\phi-\phi_\ell}{\H^{-1/2}(\Gamma_{})}.
\end{align}


\section{Numerical experiments}
\label{section:numerics}
In this section, we empirically investigate the performance of Algorithm \ref{the algorithm} in three typical situations: In Section \ref{subsec:smooth sol} and Section  \ref{subsec:singular sol}, we consider a closed boundary $\Gamma=\partial\Omega$, where the solution is smooth resp. exhibits a generic (i.e., geometry induced) singularity.
In Section \ref{subsec:slit problem}, we consider a slit problem.
In either example, the exact solution is known and allows us to compute the Galerkin error to underline reliability and efficiency of the proposed estimator.

In each example, the parametrization $\gamma$ of the part $\Gamma_{}$ of the boundary is a NURBS curve and thus has the special form
\begin{equation}
\gamma(t)=\sum_{i\in \Z} C_i R_{i,p}^{\check{\mathcal{K}}_0,\mathcal{W}_0}(t)
\end{equation}
for all $t\in [a,b]$.
Here,  $p\in \N$ is the polynomial degree, $\check{\mathcal{K}}_0$ and $\mathcal{W}_0$ are knots and weights as in Section \ref{subsec:algorithm} and $(C_i)_{i\in \Z}$ are \textit{control points} in $\R^2$ which are periodic for closed $\Gamma=\partial\Omega$.

We choose the same polynomial degree $p$ for our approximation spaces $\XX_\ell$.
Since for the refinement strategy only knot insertion is used, we can apply \eqref{eq:splines nested} and \eqref{eq:equal weight fct} to see for the first and second component of $\gamma$
\begin{align}
\gamma_1,\gamma_2 \in \mathscr{N}^p (\check{\KK}_\ell,\mathcal{W}_\ell)|_{[a,b]}.
\end{align}
Hence, this approach reflects the main idea of isogeometric analysis, where the same space is used for the geometry and for the approximation.
We compare uniform refinement, where $\mathcal{M}_{\ell}=\mathcal{N}_{\ell}$ and hence all elements are refined, and adaptive refinement with $\theta=0.75$.



\subsection{Stable implementation of adaptive IGABEM}
\begin{figure}[h] 
\psfrag{circle (Section 5.2)}{\footnotesize circle (Section 5.2)}
\psfrag{pacman (Section 5.3)}{\footnotesize pacman (Section 5.3)}
\psfrag{slit (Section 5.4)}{\footnotesize slit (Section 5.4)}
\includegraphics[width=0.45 \textwidth]{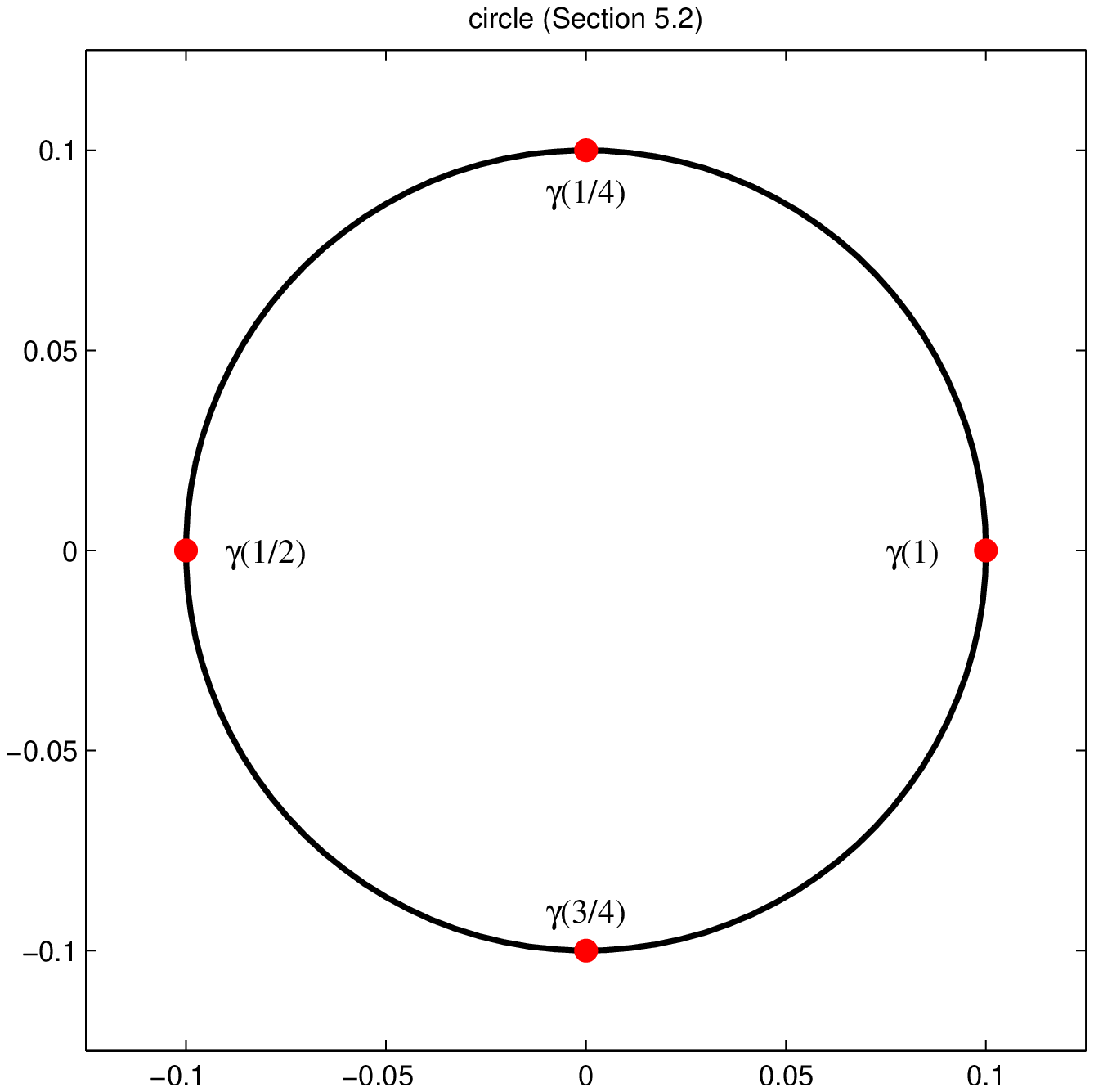}
\includegraphics[width=0.45\textwidth]{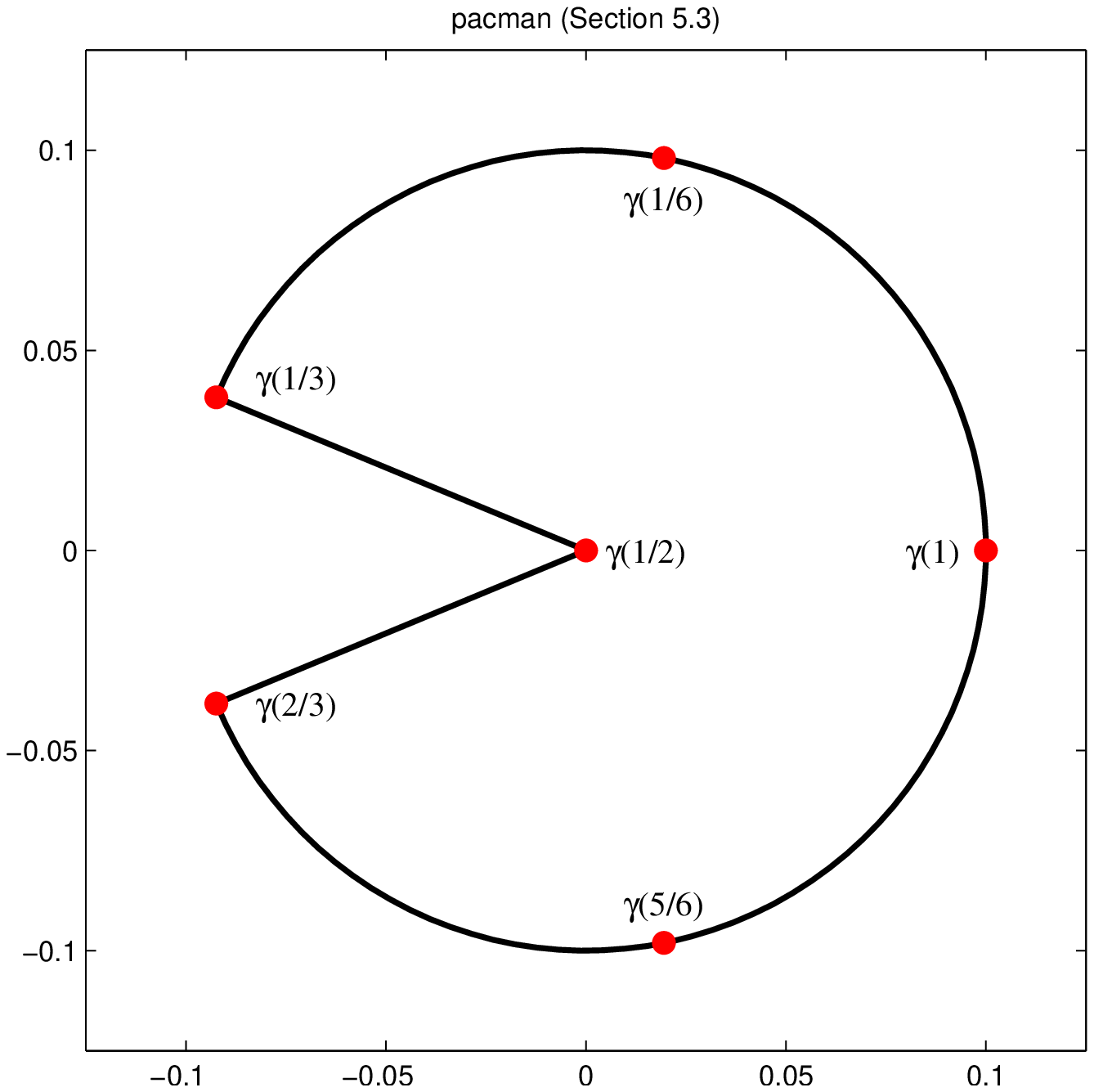}
\includegraphics[width=0.45 \textwidth]{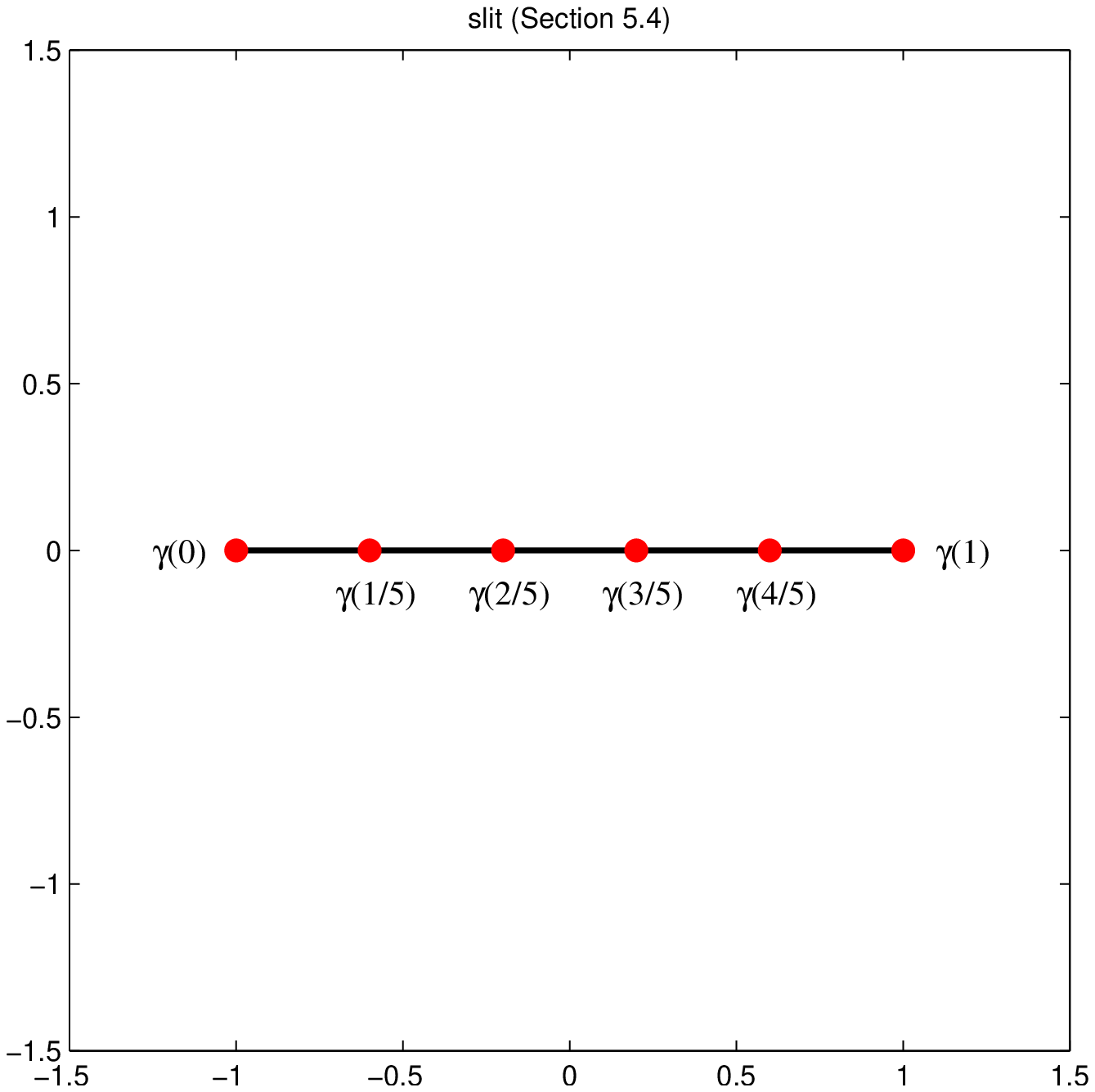}
\caption{Geometries and initial nodes for the experiments from Section~\ref{subsec:smooth sol}--\ref{subsec:slit problem}.} 
\label{fig:geometries}
\end{figure}
To compute the approximation $\phi_{h}$ of one step of the adaptive algorithm, we first note that Theorem \ref{thm:spline basis} implies that
\begin{align}
\big\{ R_{i,p}|_{[a,b)}:i=(1-p),\dots,N-\#b+1\big\}\circ\gamma|_{[a,b)}^{-1}
\end{align}
resp.
\begin{align}
\big\{ R_{i,p}|_{[a,b]}:i=1,\dots,N\big\}\circ\gamma^{-1}
\end{align}
forms a basis of $\hat{\mathscr{N}}(\check{\KK}_h,\mathcal{W}_h)$.
We abbreviate the elements of this basis with $\hat{R}_i$ and its index set with $\mathcal{I}$.
Then, there holds the unique basis representation $\phi_{h}=\sum_{i\in \mathcal{I}} c_{h,i} \hat{R}_i$.
The coefficient vector $\boldsymbol{c_h}$ is the unique solution of 
\begin{align}
\boldsymbol{V_h} \boldsymbol{c_h}= \boldsymbol{f_h}
\end{align}
 with the symmetric positive definite matrix 
\begin{align}
\boldsymbol{V_h}:=\left(\dual{V \hat{R}_j}{\hat{R}_i}_{L^2(\Gamma)}\right)_{i,j\in \mathcal{I}}
\end{align}
and the right-hand side vector
\begin{align}
\boldsymbol{f_h}:=\left(\dual{f}{\hat{R}_i}_{L^2(\Gamma)}\right)_{i\in \mathcal{I}}.
\end{align}
The energy norm then reads 
\begin{align}
\enorm{\phi_h}^2=\dual{V\phi_h}{\phi_h}=\boldsymbol{c_h}^T \boldsymbol{V_h} \boldsymbol{c_h}.
\end{align}
To calculate $\boldsymbol{V_h}$, $\boldsymbol{f_h}$ and the $H^{1/2}$-seminorms of the residual $r_h=f-V\phi_h$, singular integrals and double integrals have to be evaluated.
 Since, this is hardly possible analytically, we approximate the appearing integrals.
 To this end, we first write them as sum of integrals over the elements of the mesh $\check{\TT}$.
 In the spirit of \cite[Section 5.3]{ss}, the local integrals which contain singularities, are transformed by \textit{Duffy transformations} such that either the singularity  vanishes or a pure logarithmic singularity of the form $\log(t)$ on $[0,1]$ remains.
 Finally, the integrals are evaluated over the domain $[0,1]$ or $[0,1]^2$ using tensor-Gauss quadrature with weight function $1$ resp. $\log(t)$.
 Since the integrands are smooth up to logarithmic terms, this yields exponential convergence of adapted Gauss quadrature and hence provides  accurate approximations.
 For closed $\Gamma=\partial\Omega$ and arbitrary parametrization $\gamma$ as in Section \ref{subsec:boundary parametrization}, all details are elaborated in \cite[Section 5]{diplarbeit}.
\subsection{Adaptive IGABEM for problem with smooth solution}
\label{subsec:smooth sol}
\begin{figure}[h]
\psfrag{estimator and error}{\tiny estimator and error}
\psfrag{number of knots N}{\tiny number of knots $N$}
\psfrag{O(72)}[r][r]{\tiny$\mathcal{O}(N^{-7/2})$}
\psfrag{O(47)}[l][l]{\tiny$\mathcal{O}(N^{-4/7})$}
\psfrag{eta, unif.}[l][l]{\tiny est., unif.}
\psfrag{error, unif.}[l][l]{\tiny  error, unif.}
\psfrag{eta, ad.}[l][l]{\tiny est., ad.}
\psfrag{error, ad.}[l][l]{\tiny  error, ad.}
\includegraphics[width=0.7\textwidth]{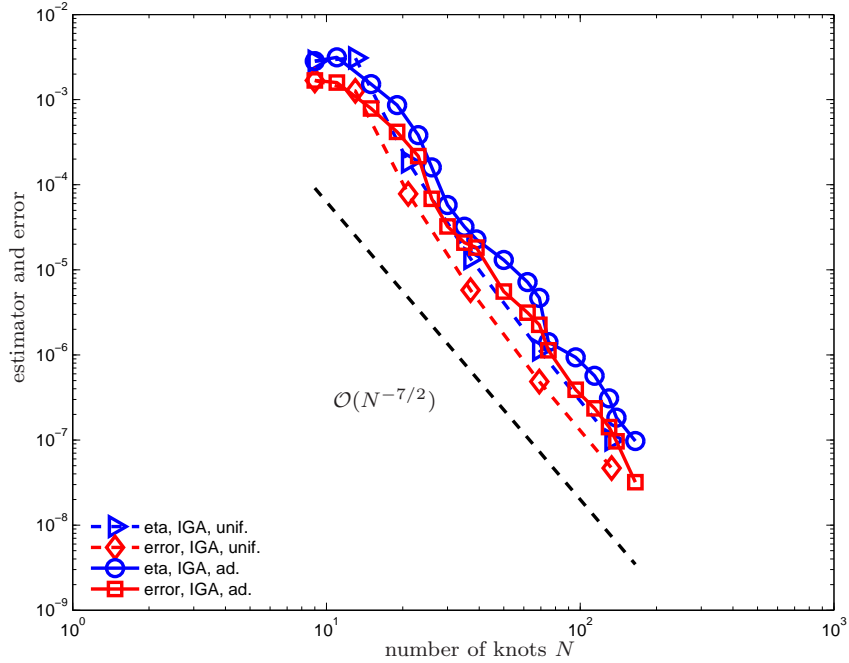}
\caption{Experiment with smooth solution on circle geometry from Section \ref{subsec:smooth sol}. Error and estimator are plotted versus the number of knots~$N$.} 
\label{fig:circle convergence}
\end{figure} 

Let $\Omega$ be the circle with midpoint $(0,0)$ and radius $1/10$.
\noindent We consider the Laplace-Dirichlet problem on $\Omega$
\begin{align}\label{eq:Laplace}
\begin{split}
-\Delta u=0\text{ in }{\Omega}\quad \text{ and }\quad u=g\text{ on } \Gamma_{}
\end{split}
\end{align}
for given Dirichlet data $g\in {H}^{1/2}(\Gamma_{})$ and closed boundary $\Gamma=\partial\Omega$.
The problem is equivalent to Symm's integral equation \eqref{eq:strong} with the \textit{single-layer integral operator}
\begin{align}\label{eq:single-layer}
V:{\H}^{-1/2}(\Gamma_{})\to H^{1/2}(\Gamma_{}),\quad V\phi(x):=-\frac{1}{2\pi}\int_{\Gamma_{}} \log (|x-y|) \phi(y) \,dy  
\end{align}
and the  right-hand side $f=(K+1/2)g$, where 
\begin{align}
K:{H}^{1/2}(\Gamma_{})\to H^{1/2}(\Gamma_{}), \quad Kg(x):= -\frac{1}{2\pi}\int_{\Gamma_{}} \big(\partial_{\nu(y)}\log (|x-y|) \big)g(y) \,dy 
\end{align}
denotes the \textit{double-layer integral operator}.
The unique solution of \eqref{eq:strong} is the normal derivative $\phi= \partial u/\partial \nu$ of the weak solution $u\in H^1(\Omega)$ of \eqref{eq:Laplace}.

We prescribe the exact solution $u(x,y)=x^2+10xy-y^2$ and solve  Symm's integral equation~\eqref{eq:strong} on the closed boundary $\Gamma=\partial\Omega$. 
The normal derivative $\phi=\partial{u}/\partial \nu$ reads
\begin{equation*}
\phi(x,y)=20 \big(x^2+10xy-y^2).
\end{equation*}
The geometry is parametrized on $[0,1]$ by the NURBS curve induced by
\begin{align*}
p&=2,\\
\check{\mathcal{K}}_0 &=\left(\frac{1}{4},\frac{1}{4},\frac{2}{4},\frac{2}{4},\frac{3}{4},\frac{3}{4},1,1,1\right),\\
\mathcal{W}_0&=\left(1,\frac{1}{\sqrt 2} ,1,\frac{1}{\sqrt 2},1,\frac{1}{\sqrt 2},1,1,\frac{1}{\sqrt 2}\right),\\
(C_i)_{i=1}^{N_0}&=\frac{1}{10}\cdot\left(\begin{pmatrix}0\\1\end{pmatrix},
\begin{pmatrix}-1\\1\end{pmatrix},
\begin{pmatrix}-1\\0\end{pmatrix},
\begin{pmatrix}-1\\-1\end{pmatrix},
\begin{pmatrix}0\\-1\end{pmatrix},
\begin{pmatrix}1\\-1\end{pmatrix},
\begin{pmatrix}1\\0\end{pmatrix},
\begin{pmatrix}1\\0\end{pmatrix},
\begin{pmatrix}1\\1\end{pmatrix}
\right).
\end{align*}
Note that this parametrization does  not coincide with the natural parametrization $t\mapsto (\cos(t),\sin(t))$.
Figure \ref{fig:geometries} visualizes the geometry and the $\gamma$-values of the initial nodes. 
Figure~\ref{fig:circle convergence} shows error and error estimator for the uniform and the adaptive approach.
All values are plotted in a log-log scale such that the experimental convergence rates are visible as the slope of the corresponding curves.
The Galerkin orthogonality allows to compute the energy error by 
\begin{align}\label{eq:error calc}
\enorm{\phi-\phi_\ell}^2=\enorm{\phi}^2-\enorm{\phi_\ell}^2=13\pi/5000-\enorm{\phi_\ell}^2,
\end{align}
With respect to the number of knots $N$, both approaches lead to the rate $\mathcal{O}(N^{-7/2})$.
If  discontinuous piecewise polynomials of order $2$ were used as ansatz space, this is the optimal convergence rate.
In each case, the curves for the error and the corresponding estimator are parallel.
This empirically confirms the proven efficiency and reliability of the Faermann estimator $\eta_h$.
\subsection{Adaptive IGABEM for problem with generic singularity}
\label{subsec:singular sol}
\begin{figure}[h] 
\psfrag{parameter domain}{\tiny parameter domain}
\psfrag{solution}{\tiny solution} 
\includegraphics[width=0.6\textwidth]{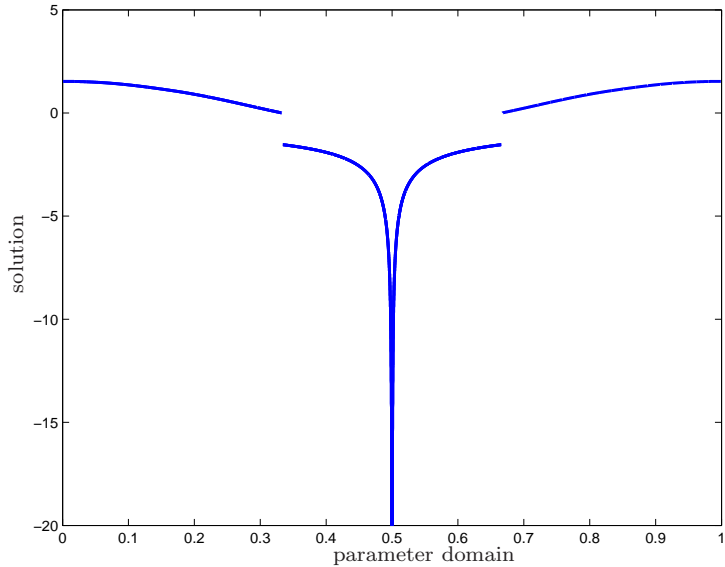}
\caption{Experiment with singular solution on pacman geometry from Section \ref{subsec:singular sol}. The singular solution $\phi\circ{\gamma}$ is plotted on the parameter interval, where $0.5$ corresponds to the origin, where $\phi$ is singular.}
\label{fig:pacman solution}
\end{figure}

\begin{figure}[h]
\psfrag{estimator and error}{\tiny estimator and error}
\psfrag{number of knots N}{\tiny number of knots $N$}
\psfrag{O(72)}[r][r]{\tiny$\mathcal{O}(N^{-7/2})$}
\psfrag{O(47)}[l][l]{\tiny$\mathcal{O}(N^{-4/7})$}
\psfrag{eta, unif.}[l][l]{\tiny est., unif.}
\psfrag{error, unif.}[l][l]{\tiny  error, unif.}
\psfrag{eta, ad.}[l][l]{\tiny est., ad.}
\psfrag{error, ad.}[l][l]{\tiny  error, ad.}
\includegraphics[width=0.7\textwidth]{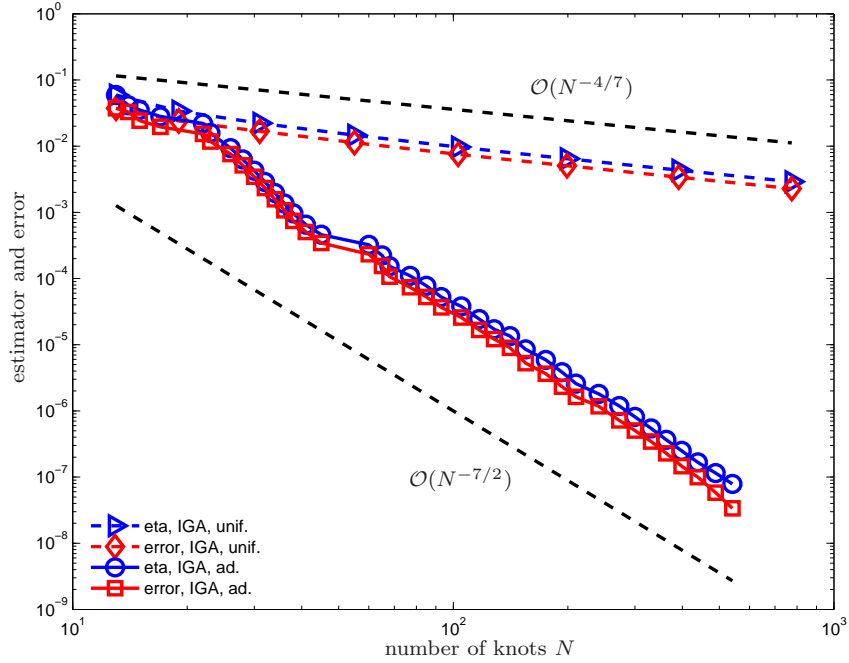}
\caption{Experiment with singular solution on pacman geometry from Section \ref{subsec:singular sol}. Error and estimator are plotted versus the number of knots~$N$.} 
\label{fig:pacman convergence}
\end{figure}

As second example, we consider the Laplace-Dirichlet problem \eqref{eq:Laplace} on the pacman geometry
\begin{equation*}
\Omega:=\left\{r(\cos(\alpha),\sin(\alpha)):0\le r<\frac{1}{10}, \alpha \in \left(-\frac{\pi}{2\tau},\frac{\pi}{2\tau}\right)\right\},
\end{equation*}
with $\tau=4/7$; see Figure \ref{fig:geometries}.
We prescribe the exact solution
\begin{equation*}
u(x,y)=r^{\tau}\cos\left(\tau\alpha\right) \quad\text{in polar coordinates}\quad (x,y)=r(\cos \alpha,\sin \alpha).
\end{equation*}
The normal derivative of $u$ reads 
\begin{equation*}
\phi(x,y)=\begin{pmatrix} \cos(\alpha)\cos\left(\tau\alpha\right)+\sin(\alpha)\sin\left(\tau\alpha\right)\\ \sin(\alpha)\cos\left(\tau\alpha\right)-\cos(\alpha)\sin\left(\tau\alpha\right)\end{pmatrix}\cdot \nu(x,y) \cdot \tau \cdot r^{\tau-1}
\end{equation*}
and has a generic singularity at the origin.
With $w=\cos(\pi/\tau)$, the geometry is parametrized on $[0,1]$ by the NURBS curve induced by
\begin{figure}[h]
\psfrag{knots}{\tiny knots}
\psfrag{indices of knots}{\tiny indices of knots}
\includegraphics[width=0.7\textwidth]{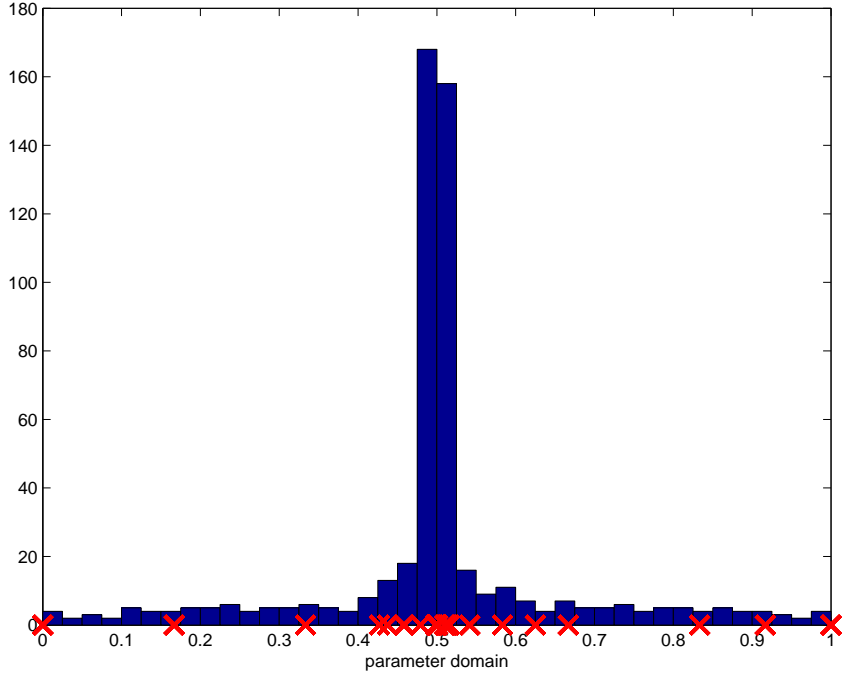}
\caption{Experiment with singular solution on pacman geometry from Section \ref{subsec:singular sol}.
Histogram of number of knots over the parameter domain. Knots with maximal multiplicity $p+1=3$ are marked.}
\label{fig:pacman knots}
\end{figure}

\begin{align*}
p&=2,\\
\check{\mathcal{K}}_0&=\left(\frac{1}{6},\frac{1}{6},\frac{2}{6},\frac{2}{6},\frac{3}{6},\frac{3}{6},\frac{4}{6},\frac{4}{6},\frac{5}{6},\frac{5}{6},1,1,1\right),\\
\mathcal{W}_0&=\left(1,w ,1,1,1,1,1,w,1,w,1,1,w\right),\\
(C_i)_{i=1}^{N_0}&=\frac{1}{10}\cdot\Bigg(
\begin{pmatrix}\cos(\pi/\tau\cdot 2/8)\\ \sin(\pi/\tau\cdot 2/8)\end{pmatrix},
\frac{1}{w}\begin{pmatrix}\cos(\pi/\tau\cdot 3/8)\\ \sin(\pi/\tau\cdot 3/8)\end{pmatrix},
\begin{pmatrix}\cos(\pi/\tau\cdot 4/8)\\ \sin(\pi/\tau\cdot 4/8)\end{pmatrix},\\
&\quad\frac{1}{2}\begin{pmatrix}\cos(\pi/\tau\cdot 4/8)\\ \sin(\pi/\tau\cdot 4/8)\end{pmatrix},
\begin{pmatrix}0\\ 0\end{pmatrix},
\frac{1}{2}\begin{pmatrix}\cos(\pi/\tau\cdot (-4)/8)\\ \sin(\pi/\tau\cdot (-4)/8)\end{pmatrix},
\begin{pmatrix}\cos(\pi/\tau\cdot (-4)/8)\\ \sin(\pi/\tau\cdot (-4)/8)\end{pmatrix},\\
&\quad\frac{1}{w}\begin{pmatrix}\cos(\pi/\tau\cdot (-3)/8)\\ \sin(\pi/\tau\cdot (-3)/8)\end{pmatrix},
\begin{pmatrix}\cos(\pi/\tau\cdot (-2)/8)\\ \sin(\pi/\tau\cdot (-2)/8)\end{pmatrix},
\frac{1}{w}\begin{pmatrix}\cos(\pi/\tau\cdot (-1)/8)\\ \sin(\pi/\tau\cdot (-1)/8)\end{pmatrix},\\
&\quad \begin{pmatrix}\cos(\pi/\tau\cdot 0/8)\\ \sin(\pi/\tau\cdot 0/8)\end{pmatrix}, 
\begin{pmatrix}\cos(\pi/\tau\cdot 0/8)\\ \sin(\pi/\tau\cdot 0/8)\end{pmatrix},
\frac{1}{w}\begin{pmatrix}\cos(\pi/\tau\cdot 1/8)\\ \sin(\pi/\tau\cdot 1/8)\end{pmatrix}
\Bigg).
\end{align*}

In Figure \ref{fig:pacman solution}, the solution $\phi$ is plotted over the parameter domain.
We can see that $\phi$ has a singularity at $t=1/2$ as well as two jumps at  $t=1/3$ resp. $t=2/3$.

  In Figure \ref{fig:pacman convergence},  error and error estimator are plotted.
 As the respective curves are parallel, we empirically confirm efficiency and reliability of the Faermann estimator.
For the calculation of the error, we used $\enorm{\phi}^2=0.083525924784082$ in \eqref{eq:error calc} which is obtained by Aitkin's $\Delta^2$-extrapolation.
Since the solution lacks regularity, uniform refinement leads to the suboptimal rate $\mathcal{O}(N^{-4/7})$, whereas adaptive refinement leads to the optimal rate $\mathcal{O}(N^{-7/2})$.

For adaptive refinement, Figure \ref{fig:pacman knots} provides a histogram of the knots in $[a,b]$ of the last refinement step.
We see that the algorithm mainly refines the mesh around the singularity at $t=1/2$.
Moreover, the multiplicity at the jump points $t=1/3$ and $t=2/3$ appears to be maximal so that the discrete solution $\phi_\ell$  also mimics the discontinuities of the exact solution $\phi$. 
Hence the functions of the considered ansatz space do not need to be continuous there, see Theorem \ref{thm:spline basis}.

\subsection{Adaptive IGABEM for slit problem}
\label{subsec:slit problem}

\begin{figure}[h]
\psfrag{estimator and error}{\tiny estimator and error}
\psfrag{number of knots N}{\tiny number of knots $N$}
\psfrag{O(52)}[r][r]{\tiny$\mathcal{O}(N^{-5/2})$}
\psfrag{O(12)}[l][l]{\tiny$\mathcal{O}(N^{-1/2})$}
\psfrag{eta, unif.}[l][l]{\tiny est., unif.}
\psfrag{error, unif.}[l][l]{\tiny  error, unif.}
\psfrag{eta, ad.}[l][l]{\tiny est., ad.}
\psfrag{error, ad.}[l][l]{\tiny  error, ad.}
\includegraphics[width=0.7\textwidth]{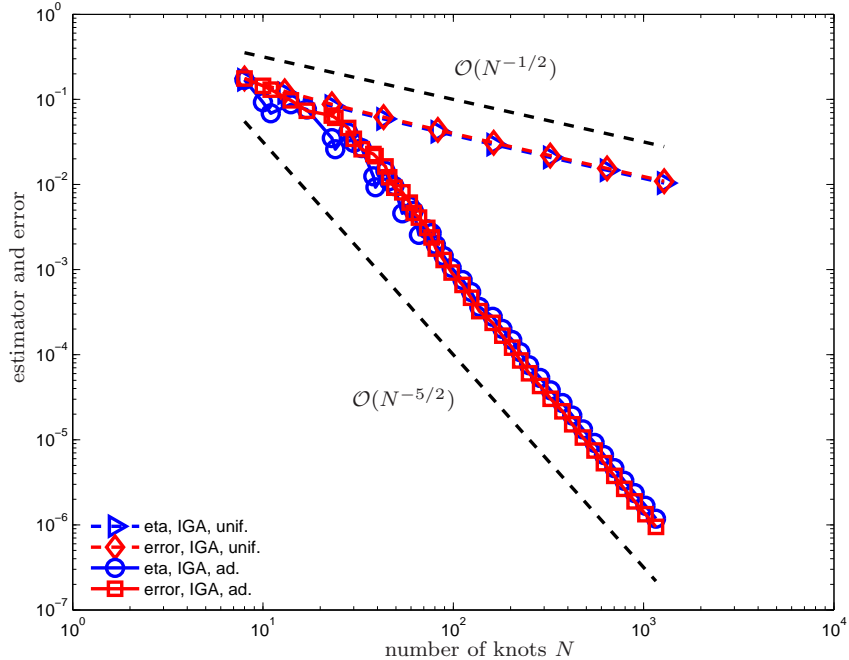}
\caption{Experiment with singular solution on slit from SectionÊ \ref{subsec:slit problem}. Error and estimator are plotted versus the number of knots~$N$.} 
\label{fig:slit convergence}
\end{figure} 

\begin{figure}[h]
\psfrag{knots}{\tiny knots}
\psfrag{indices of knots}{\tiny indices of knots}
\includegraphics[width=0.7\textwidth]{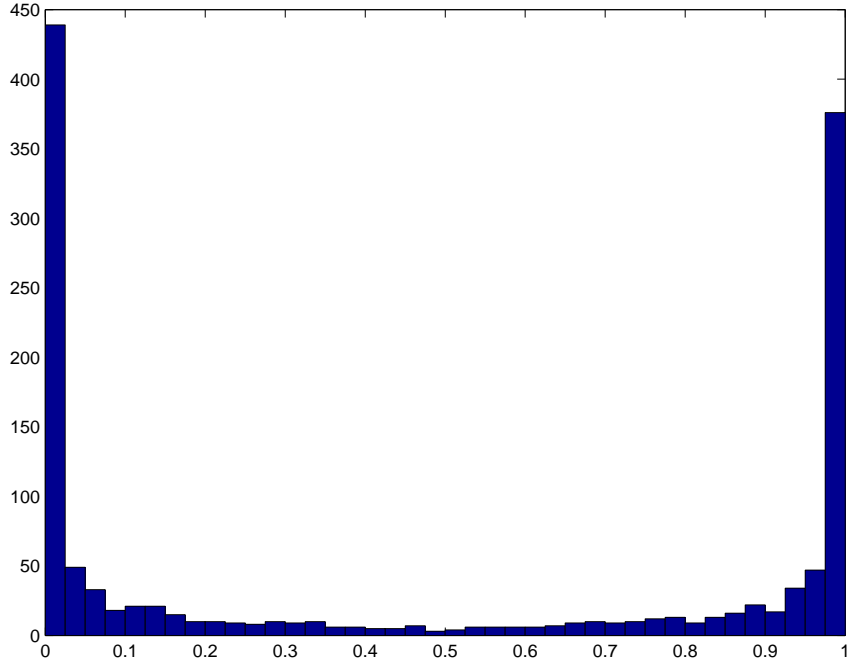}
\caption{Experiment with singular solution on slit from SectionÊ \ref{subsec:slit problem}. 
Histogram of number of knots over the parameter domain. } 
\label{fig:slit knots}
\end{figure}

As last example, we consider a crack problem on the slit $\Gamma= [-1,1]\times\{0\}$. 
 For $f(x,0):=-x/2$ and the single-layer operator $V$ from \eqref{eq:single-layer}, the exact solution of \eqref{eq:strong} reads
\begin{align*}
\phi(x,0)=\frac{-x}{\sqrt{1-x^2}}.
\end{align*}
Note that $\phi\in \H^{-\varepsilon}(\Gamma)\setminus L^2(\Gamma)$ for all $\varepsilon>0$ and that $\phi$ has singularities at the tips $x=\pm 1$.
We parametrize $\Gamma$ by the NURBS curve induced by

\begin{align*}
p&=1,\\
\check{\mathcal{K}}_0&=\left(0,0,\frac{1}{5},\frac{2}{5},\frac{3}{5},\frac{4}{5},1,1\right),\\\mathcal{W}_0&=\left(1,1,1,1,1,1\right),\\
(C_i)_{i=1}^{N_0-p}&=\left(
\begin{pmatrix}-1\\0\end{pmatrix},
\begin{pmatrix}-3/5\\0\end{pmatrix},
\begin{pmatrix}-1/5\\0\end{pmatrix},
\begin{pmatrix}1/5\\0\end{pmatrix},
\begin{pmatrix}3/5\\0\end{pmatrix},
\begin{pmatrix}1\\0\end{pmatrix}\right).
\end{align*}

In Figure \ref{fig:slit convergence},   error and error estimator for the uniform and  for the adaptive approach are plotted.
The error is obtained via \eqref{eq:error calc}, where $\enorm{\phi}^2=\pi/4$ is computed analytically.
Since the solution lacks regularity, uniform refinement leads to the suboptimal rate $\mathcal{O}(N^{-1/2})$, whereas adaptive refinement leads to the optimal rate $\mathcal{O}(N^{-5/2})$.

For adaptive refinement, we plot in Figure \ref{fig:slit knots}  a histogram of the knots in $[a,b]=[0,1]$ of the last refinement step.
As expected, the algorithm mainly refines the mesh at the tips  $t=0$ and $t=1$.


\bigskip

\noindent
{\bf Acknowledgement.} The authors acknowledge support through the Austrian Science 
Fund (FWF) under grant P21732 \emph{Adaptive Boundary Element Method} as well as
P27005 \emph{Optimal adaptivity for BEM and FEM-BEM coupling}. In addition, DP and MF
are supported through the FWF doctoral school \emph{Nonlinear PDEs} funded under
grant W1245, and GG through FWF under grant P26252 \emph{Infinite elements for exterior Maxwell problems}.

\bibliographystyle{alpha}
\bibliography{literature}

\end{document}